\documentclass[leqno]{amsart}
\usepackage{amsfonts,amssymb,amsmath,amsgen,amsthm}
\usepackage{hyperref}
\usepackage{graphicx}

\theoremstyle{plain}
\newtheorem{theorem}{Theorem}[section]
\newtheorem{definition}{Definition}[section]

\newtheorem{lemma}[definition]{Lemma}

\newtheorem{proposition}[definition]{Proposition}

\newtheorem{remark}[definition]{Remark}

\usepackage{color,ulem,textcomp}

\def\Tend#1#2{\mathop{\longrightarrow}\limits_{#1\rightarrow#2}}

 \font\tenms=msbm10
 \font\tenms=msbm10

\font\sevenms=msbm7 \font\fivems=msbm5
\newfam\bbfam \textfont\bbfam=\tenms \scriptfont\bbfam=\sevenms
\scriptscriptfont\bbfam=\fivems

\def\build#1_#2^#3{\mathrel{
\mathop{\kern 0pt#1}\limits_{#2}^{#3}}}

\newcommand{\beq}{\begin{eqnarray}}
\newcommand{\eeq}{\end{eqnarray}}
\newcommand{\bq}{\begin{equation}}
\newcommand{\eq}{\end{equation}}
\newcommand{\beqn}{\begin{eqnarray*}}
\newcommand{\eeqn}{\end{eqnarray*}}

\let\lam=\lambda

\def\DD{\mathop{\bf D\kern 0pt}\nolimits}

\def\SS{\mathop{\bf S\kern 0pt}\nolimits}
\def\ZZ{\mathop{\bf Z\kern 0pt}\nolimits}
\def\TT{\mathop{\bf T\kern 0pt}\nolimits}
\def\virgp{\raise 2pt\hbox{,}}
\def\cdotpv{\raise 1pt\hbox{ ;}}

\def\eps{\varepsilon}
\def\beq{\begin{equation}}
\def\eeq{\end{equation}}

\def\cdotv{\raise 2pt\hbox{,}}

\def\H{\mathop{\mathbb H\kern 0pt}\nolimits}
\def\R{\mathop{\mathbb R\kern 0pt}\nolimits}

\def\C{{\mathbb C}}
\def\N{{\mathbf N}}
\def\Q{{\mathbf Q}}

\def\virgp{\raise 2pt\hbox{,}}

\def\({\left(}
\def\){\right)}
\def\<{\left\langle}
\def\>{\right\rangle}

\def\Tend#1#2{\mathop{\longrightarrow}\limits_{#1\rightarrow#2}}

\numberwithin{equation}{section}

\begin{document}
\title[]{Semi-classical analysis on H-type groups}
\author[C. Fermanian]{Clotilde~Fermanian-Kammerer}
\address[C. Fermanian Kammerer]{
Universit\'e Paris Est Cr\'eteil, LAMA, 61, avenue du G\'en\'eral de Gaulle\\
94010 Cr\'eteil Cedex\\ France}
\email{clotilde.fermanian@u-pec.fr}
\author[V. Fischer]{V\'eronique Fischer}\address[V. Fischer]%
{University of Bath, Department of Mathematical Sciences, Bath, BA2 7AY, UK} 
\email{v.c.m.fischer@bath.ac.uk}

\begin{abstract}
In this paper, we develop a semi-classical analysis on H-type groups. We define semi-classical pseudodifferential operators, prove the boundedness of their action on square integrable functions and develop a symbolic calculus. Then, we define the semi-classical measures of bounded families of square integrable functions which consists of a pair formed by a measure defined on the product of the group and its unitary dual, and by a field of trace class positive operators acting on the Hilbert spaces of the representations. We illustrate the theory by analyzing examples, which show in particular that this semi-classical analysis takes into account the finite-dimensioned representations of the group, even though they are negligible with respect to the Plancherel measure. 
 \end{abstract}
\thanks{}
\maketitle

\section{Introduction}\label{intro}

%

The H-type groups are connected simply connected Lie group  $G$ which are 
 stratified of step~$2$  as the Heisenberg group. 
 There are simple relationships between the dual of the centrum and the first strata which makes them very good toys-model for testing   ideas and building examples. 
 Our aim here is to develop semi-classical analysis on such groups by introducing semi-classical pseudo-differential operators and semi-classical Wigner measures in the spirit of previous works~\cite{BFG1,BFG2,FR,FF}.
 
 \medskip
 
 Semi-classical pseudodifferential operators are basic tools for semi-classical analysis, 
 and their development is related with this of microlocal analysis. 
 Microlocal analysis (which uses classical pseudodifferential operators, without small parameter) was mainly developed in the 70's. 
 The funding idea is to study phenomenon simultaneously in standard and Fourier variables, 
 which correspond to position and impulsion variables, 
 the phase space variables of quantum mechanics. 
 Therefore, with such a view point, it is crucial to be equipped with localization operators in position and impulsion, 
 and that is what makes possible the pseudodiffrential operators introduced by Shubin~\cite{Shu} and systematically studied by H\"ormander~\cite{Hor}. 
In his  annotated bibliography, 
Bernard Helffer~\cite{He} dates semi-classical earliest results to Andr\'e Voros thesis where one sees the first use of pseudodifferential calculus in a semi-classical context. 
The question is then to use microlocal technics on problems  where a small parameter is present and to analyze microlocal properties when this small parameter goes to~$0$.

\medskip

The transposition of the aforementioned point of view in the setting of Lie groups has been a subject of investigation  where microlocal approach or phase-space analysis, has been developed in the context of the Heisenberg group (see \cite{BGX,BCX1,BCG} for example). Such strategies are faced to the difficulty of the operator structure of the Fourier transform and require to use a non-commutative setting. 
Various attempts have been made to construct a pseudodifferential calculus since the 80's. 
One can cite the pioneer works of~\cite{Taylor84,BealsGreiner,Geller90,CGGP}. 
We also refer to~\cite{Ruz} on compact Lie groups and to the introductions of~\cite{BFG1,FR} for an overview on the subject. More recently, one of the authors and her colleagues have developed such a calculus in the context of the Heisenberg group in~\cite{BFG1}. The full program has been achieved by the second author in~\cite{FR} where pseudodifferential  operators are defined on graded Lie groups. 
We define here  semi-classical operators in the framework of H-type groups, using  the general theory developed in the preceding references and we develop semi-classical measures in this very context, as we shall see in the next sections. 

\medskip 

Semi-classical  measures, also called Wigner measures, 
were used throughout  the 90's, in particular in articles~\cite{LionsPaul} and~\cite{gerardleichtnam} (see also~\cite{gerard_X} and~\cite{GMMP}). 
They have opened an elegant path to study the compactness defect of sequences of functions which have a special size of oscillations. 
For that reason, they have  contributed to prove important results on the density of families of eigenfunctions of the Laplacian on the torus (cf.~\cite{AM:12} and~\cite{AM:14});
 it is indeed possible to link semi-classical measures with weak limits of densities of wave functions. 
 The authors think that, once suitably generalized in the framework of Lie groups, these technics should be useful for studying quantum ergodicity in sub-Riemaninan geometries, in the spirit of the recent result~\cite{CdVHT}. 

\medskip

Of course, the approach developed here adapt to more general graded Lie groups. 
However, we choose the simple case of H-type groups in order to illustrate that, in this simple example, the microlocal approach developed here takes into account tricky questions linked for example with representations. 
Indeed, despite  the fact that finite dimensional representations are of null measure with respect to the Plancherel  measure, 
they are taken into account by the averaging process implemented in the computation of semi-classical measures. 
We illustrate this fact with examples which should be answer the questions of interest to Fulvio Ricci and asked by Bernard Helffer to the authors. 
These examples also emphasize the large complexity of the Fourier space in H-type groups (see~\cite{BCD} for example).

\medskip

In the next section of this paper, we recall basic facts about H-type groups. Then, we introduce the algebra of  semi-classical pseudo-differential operators and describe its main properties, we define semi-classical measures and introduce a notion of Wigner distributions. Finally, we analyze the semi-classical measures of several families of functions, emphasizing concentration effects in space variables, and oscillations properties that we interpret as concentration in Fourier variables, including concentration on finite-dimensional representations.

 
 \section{H-type groups}
 
 A connected simply connected Lie group  $G$  is said to be stratified of step~$2$
  if its left-invariant Lie algebra~${\mathfrak g}$ (assumed  real-valued and  of finite dimension~$n$)    is  endowed with a vector space decomposition
$$  
  \displaystyle
  {\mathfrak g}=  \mathfrak v \oplus \mathfrak z \, ,
  $$
  such that $[{\mathfrak v},{\mathfrak v}]= {\mathfrak z}$ and ${\mathfrak z}$ is the center of ${\mathfrak g}$. Via the exponential map  
 $$
 {\rm exp} :  {\mathfrak g} \rightarrow G 
 $$ which is in that case a diffeomorphism from ${\mathfrak g}$ to $G$, one identifies the  sets of $G$ and ${\mathfrak g}$ with the underlying vector space. It turns out that under this identification, the  group law on $G$ (which is generally not commutative) provided by the Campbell-Baker-Hausdorff formula, $(x,y)  \mapsto x  y  $ is a polynomial map.  More precisely, if $x={\rm Exp}(v_x+z_x)$ and 
 $y={\rm Exp}(v_y+z_y)$ then
 $$xy ={\rm Exp} (v+z),\;\; v=v_x+v_y\in{\mathfrak v},\;\;z= z_x+z_y+\frac 12 [v_x,v_y]\in{\mathfrak z}.$$
If $x={\rm Exp}(v)$ then $x^{-1}={\rm Exp}(-v)$.

    \medskip
    
        For any~$\lambda \in  \mathfrak z^\star$ (the dual of the center~$  \mathfrak z$) we define a skew-symmetric bilinear form on $\mathfrak v$ by 
\begin{equation}\label{skw}
\forall \, U,V \in  \mathfrak v \, , \quad B(\lambda) (U,V):= \lambda([U,V]) \, .
 \end{equation}
 Following~\cite{Kaplan80}, we say that $G$ is of H-type (or Heisenberg type) if, once fixed an inner product on~${\mathfrak v}$ and on~${\mathfrak z}$, and in any orthonormal basis  of~${\mathfrak v}$, the endomorphism of this skew symmetric form (that we still denote by $B(\lambda)$) satisfies 
 $$\forall \lambda\in{\mathfrak z}^*,\;\; B(\lambda)^2=-|\lambda|^2 {\rm Id}_{\mathfrak v}.$$
 This implies in particular that the dimension of ${\mathfrak v}$ is even. We set 
 $${\rm dim}\, {\mathfrak v} =2d,\;\; {\rm dim} \,{\mathfrak z}=p.$$
 Naturally, a Heisenberg group is an H-type group; more precisely, an H-type group is Heisenberg if and only if the dimension of the centre $p$ is equal to 1. Recall that the Heisenberg group $\mathbb H_d$ is the set $\mathbb R^{2d+1}$ equipped with the following product
 $$
 hh'= (p+p',q+q',z+z'+\frac 12 (pq'-p'q), \quad\mbox{where} \quad
 h=(p,q,z),h'=(p',q',z')\in \R^d \times \R^d \times \R;
 $$
 here $pq'$ denotes the standard inner product of the two vectors $p,q'\in \R^d$.
 Its Lie algebra is the vector space $\mathbb R^{2d+1}$ equipped with the following Lie bracket
 $$
 [(p,q,z),(p',q',z')]=(0,0,pq'-p'q). $$
 The first stratum and the centre are  
 $\mathfrak v=\R^d\times \R^d\times \{0\}\sim \R^d\times \R^d$ and 
 $\mathfrak z=\{0\}\times \{0\}\times \R$. 
 For $\lambda\in \mathfrak z^*$, the skew-symmetric bilinear form on $\mathfrak v$ is  
   \begin{equation}
\label{eqBJ_Hd}  
B(\lambda)((p,q),(p',q')) = \lambda
\begin{pmatrix} p \\ q\end{pmatrix}^t J
\begin{pmatrix} p' \\ q'\end{pmatrix}, 
\quad\mbox{where}\quad 
J=\begin{pmatrix}0 & {\rm Id} \\ -{\rm Id} & 0\end{pmatrix}.
  \end{equation}

 \medskip 

We fix an inner product on ${\mathfrak z}$, this allows us to consider the Lebesgue measure $dv\, dz$ on ${\mathfrak g}={\mathfrak v}\oplus{\mathfrak z}$. Via the identification of $G$ with ${\mathfrak g}$ via the exponential map, this induces a Haar measure $dx$ on $G$. This measure is invariant under left and right translations:
  $$
  \forall f  \in L^1(G, d\mu) \, ,  \quad  \forall x  \in G \,, \quad \int_G f(y) dy  = \int_G f(x  y)dy= \int_G f(y  x)dy \, .
   $$
Note that  the convolution of two functions $f$ and $g$ on $G$ is given by
    \begin{equation}
\label{convolutiondef}
  f*g(x) :=  \int_G f(x  y^{-1})g(y)dy = \int_G f(y)g(y^{-1}   x)dy,
   \end{equation}
  and  as in the Euclidean case we define   Lebesgue spaces by
$$
 \|f\|_{L^q (G)}  := \left( \int_G |f(y)|^q \: dy \right)^\frac1q \, ,
 $$
 for $q\in[1,\infty)$, with the standard modification when~$q=\infty$.\\

\medskip

  Since $G$ is stratified,   there is a natural family of dilations on ${\mathfrak g}$ defined for $t>0$ as follows: if~$X$ belongs to~$ {\mathfrak g}$, we can decompose~$X$ as~$\displaystyle X=V+Z$ with~$V\in {\mathfrak v}$ and~$Z\in {\mathfrak z}$, then
   $$
   \delta_t X:=tV+t^2Z  \, .
   $$
 This allows us to  define the dilation on the Lie group $G$ via the identification by the exponential map:
 $$\begin{array}{ccccc}
& {\mathfrak g} &\build{\rightarrow}_{}^{\delta_t} & {\mathfrak g}&\\
 {\small\rm exp}&  \downarrow& & \downarrow& {\small\rm exp}\\
  &G &\build{\rightarrow}_{ {\rm exp}\, \circ\,  \delta_t \, \circ\,  {\rm exp}^{-1}}^{}&G
  \end{array}$$ 
 To avoid heavy notations,  
 we shall still denote by $\delta_t$ the map ${\rm exp}\, \circ \delta_t \, \circ {\rm exp}^{-1}$.
 The dilations $\delta_t$, $t>0$, on $\mathfrak g$ and $G$ form a one-parameter group of automorphisms of the Lie algebra $\mathfrak g$ and of the group $G$.
The Jacobian of the dilation $\delta_t$ is $t^Q$ where
 $$Q:={\rm dim}\, {\mathfrak v} +2{\rm dim}\, {\mathfrak z} = 2d+2p$$
  is called the homogeneous dimension of $G$:
   \begin{equation}\label{homogenedim} \int_G f(\delta_t\,y) \,dy  =  t^{-Q}\,\int_G f( y)\,dy \, .
   \end{equation}
 We may identify ${\mathfrak g}$ with the space of left-invariant vector field via 
 $$Xf= \left.{d\over dt} f({\rm Exp}(tX)\right|_{t=0}.$$
 A differential operator $T$ on $G$
(and more generally any operator $T$ defined on $C^\infty_c(G)$ and valued in the distributions of $G\sim \R^{2d+p}$) 
 is said to be homogeneous of degree $\nu$ (or $\nu$-homogeneous) when 
 $$
 T (f\circ \delta_t) = t^\nu (Tf)\circ \delta_t. 
 $$
 For instance, a left invariant vector field in ${\mathfrak v}$ is 1-homogeneous while a left-invariant vector field in ${\mathfrak z}$  is 2-homogeneous.

\medskip 

We can   define the Schwartz space~${\mathcal S}(G)$   as the set of smooth functions on~$G$ such that
 for all~$ \alpha,\beta$ in~${\mathbb N}^{2d+p}$,  
 the function
 $ x\mapsto x^\beta   {\mathcal X}^{\alpha}f(x) $ belongs to~$ L^\infty(G),
  $ where~${\mathcal X}^{\alpha}$ denotes a product of~$|\alpha|$  left invariant vector fields forming a basis of~${\mathfrak g}$ and $x^\beta$ a product of $|\beta|$ coordinate functions on $G\sim \mathfrak v \times \mathfrak z$. The Schwartz space~${\mathcal S}(G)$ has properties very similar to those of the Schwartz space~${\mathcal S}(\R^{2d+p})$, particularly density in Lebesgue spaces.

\subsection{The Fourier transform}\label{Fourier}  
The group $G$ being non commutative, its Fourier transform is defined by means of  irreducible unitary representations.  

\subsubsection{Irreducible unitary representations}\label{defirreducible} We assume $\lambda\in \mathfrak z^*\setminus\{0\}$ and  use the skew-symmetric bilinear form defined on~\eqref{skw}.
 One can find an orthonormal basis
  $\left (P_1 , \dots ,P_d,  Q_1 , \dots ,Q_d\right)$ 
  where $B(\lambda)$ is represented by the matrix $|\lambda|J$, 
  that is, 
    \begin{equation}
\label{eqBJ}  
  B(\lambda)(U,V)= |\lambda| U^t JV,
  \end{equation}
for two vector $U,V\in \mathfrak v$ written in the $\left (P_1 , \dots ,P_d,  Q_1 , \dots ,Q_d\right)$-basis;
recall that the matrix $J$ was defined in \eqref{eqBJ_Hd}.
For example, on the Heisenberg group $\mathbb H_d$, 
 in view of \eqref{eqBJ_Hd},
  the canonical coordinates $(p,q)$ yield an orthonormal basis where  \eqref{eqBJ} holds for $\lambda>0$; however, for $\lambda<0$ this needs modifying. 
  
  \medskip 
  
We decompose~$ \mathfrak v$ in a $\lambda$-depending way as
$ \mathfrak v = \mathfrak p_\lambda\oplus  \mathfrak q_\lambda$
 with 
 $$
 \begin{aligned}
  \mathfrak p:=\mathfrak p_\lambda:= \mbox{Span} \, \big (P_1, \dots ,P_d \big) \, , & \quad \mathfrak q:=\mathfrak q_\lambda:= \mbox{Span} \, \big (Q_1, \dots ,Q_d\big).
   \end{aligned}
$$
We shall denote by $p=(p_1,\cdots,p_d)$ the coordinates of $P$ on the vector basis $(P_1, \cdots,P_d)$, by $q=(q_1,\cdots,q_d)$ those of $Q$ on $(Q_1,\cdots,Q_d)$ and by $z=(z_1,\cdots ,z_p)$ those of $Z$ on a basis $(Z_1,\cdots, Z_p)$. 
We will often use the writing of an element $x\in G$ or $X\in \mathfrak g$ as 
\begin{equation}
\label{eqxpqz}
x={\rm Exp}(X) , \qquad X=p_1P_1 +\ldots + p_dP_d \ + \ q_1Q_1 + \ldots + q_d Q_d \ + \ z_1 Z_1 +\ldots + z_p Z_p.
\end{equation}
We introduce irreducible unitary representations $\pi^{\lambda}_{x} $ of~$G$ on~$L^2( \mathfrak p_\lambda)$ by 
\begin{equation}\label{def:pilambdanu}
\pi^{\lambda}_{x} \Phi(\xi)=
 {\rm exp}\left[{i\lambda(z)+ \frac i2 |\lambda|\,p q +i\sqrt{|\lambda|} \,\xi q} \right]\Phi \left(\xi+\sqrt{|\lambda|}p\right),
 \end{equation}
 where $x$ has been written as in \eqref{eqxpqz}.
 Note that, setting 
\begin{equation}\label{opT}
T_r\Phi(\xi)=r^{d}\Phi\left(r\xi_1,\cdots,r\xi_d\right),\;\;\forall \Phi \in L^2( \mathfrak p_\lambda),
\end{equation}
we obtain for each $r>0$ a unitary operator $T_r$ on $L^2(\mathfrak p_\lambda)$
and it satisfies $T_r^*=T_{1/r}=T_r^{-1}$. Furthermore, intertwining $\pi^\lambda$ with $T_r$ when $r=\sqrt{|\lambda|}$ yields
\begin{equation}\label{defpilambda}
T_{\sqrt{|\lambda|}} \pi^{\lambda}_{x}T^*_{\sqrt{|\lambda|}}  \Phi(\xi) := {\rm Exp}\left[i\lambda (Z+ [\xi +\frac12 P , Q]) \right]\Phi ( \xi+P),\;\; \forall \Phi \in L^2( \mathfrak p_\lambda).
\end{equation}

\begin{remark} [Link with the Wigner transform] 
Identifying $\mathfrak p_\lambda\sim \R^d$, 
For $f,g\in L^2(\R^d)$, we have 
$$
\left(\pi^{\lambda}_x f,g\right) 
= {\rm e}^{i\lambda(z)} 
W[f,g]\left(\sqrt{|\lambda|}p,\sqrt{|\lambda|}q\right),$$
where for $p,q \in\R^d$ and $\eps>0$
$$
W[f,g] (p,q) =\int_{\R^d} {\rm e}^{iqv} f\left(v+{1\over 2} p\right) \ \overline{g}\left(v-{1\over 2} p\right) dv.
$$
\end{remark}

\medskip 

 The representations $\pi^\lambda$, $\lambda\in \mathfrak z^*\setminus\{0\}$, are infinite dimensional.
  There also exist other unitary irreducible representations of $G$.
 They are given by the characters of the first stratum in the following way:
 for every  $\omega\in \mathfrak v ^*$, 
  we set
$$\pi^{0,\omega}_x= {\rm e}^{i \omega(V)}, \quad
x={\rm Exp} (V+Z)\in G, \quad\mbox{with}\ V\in{\mathfrak v} \ \mbox{and} \  Z\in{\mathfrak z},
$$ 
and this defines a 1-dimensional representation $\pi^{0,\omega}$ of $G$.
Up to unitary equivalence, there are no other unitary irreducible representations of $G$ than $\pi^\lambda$, $\lambda\in \mathfrak z^*\setminus\{0\}$ and $\pi^{0,\omega}$, $\omega\in \mathfrak v^*$. 
This can be proved in two different ways.

\medskip 

The first way is to consider an irreducible unitary (non-trvial) representation $\pi$ of $G$,  to quotient the group and the representation by $\ker \pi$ and to compare the latters with the well-known representations of~$\mathbb H_d$
 and with the characters of $\mathbb R^{2d}$.
Indeed, the restriction of $\pi$ to the center $Z={\rm Exp}(\mathfrak z)$ of $G$ must be a character of $Z$ (as $\pi(z)$, $z\in Z$, intertwines $\pi$), so it
 is of the form $\pi({\rm Exp}(z)) = e^{i\lambda (z)}$ for some $\lambda \in \mathfrak z^*$. 
 If $\lambda=0$, then one can see that $G/\ker
\pi$ is isomorphic to an abelian subgroup of $\mathfrak v\sim \mathbb R^{2d}$ and that the quotient of $\pi$ is therefore a character $e^{i \omega(\cdot)}$ of this group.
If $\lambda\not=0$, then one uses that $G/\ker \pi$ is isomorphic to the Heisenberg group $\mathbb H_d$ and that the quotient of $\pi$ by $\ker \pi$ becomes via this isomorphism a  representation of $\mathbb H_d$ whose restriction on the centre is $e^{i\lambda (\cdot)}$; 
this characterises the $\lambda$-Schr\"odinger representation of $\mathbb H_d$ by the Stone - Von Neumann Theorem.

\medskip

The second way is to use Kirillov's theory. Note that this theory is shown recursively using quotients by kernels and eventually the Stone - Von Neumann Theorem; so, although Kirillov's theory is a general and powerful tool, the two proofs are in fact readily equivalent for groups of Heisenberg type.
By Kirillov's theory, the set of unitary irreducible representations modulo unitary equivalence, that we shall denote by $\widehat G$, 
is in bijection with the orbit of $\mathfrak g^*$ under the co-adjoint action of $G$.
For instance, one can compute that $\lambda\in \mathfrak z^*\setminus\{0\}$ corresponds to the class of $\pi^\lambda$
while $\omega \in \mathfrak v ^*$ corresponds to the class of $\pi^{0,\omega}$.
Simple calculations show that any two co-adjoint orbits of each
$\lambda\in \mathfrak z^*\setminus\{0\}$ and also of each $\omega \in \mathfrak v ^*$ are disjoint and that the (disjoint) union of all these orbits is~$\mathfrak g^*$.
Consequently, $\widehat G$, can be parametrized by $({\mathfrak z}^*\setminus \{0\})\sqcup {\mathfrak v}^*$: 
\begin{equation}
\label{eq_widehatG}	
\widehat G = 
\{\mbox{class of} \ \pi^\lambda \ : \ \lambda \in \mathfrak z^* \setminus\{0\}\} 
\sqcup \{\mbox{class of} \ \pi^{0,\omega} \ : \ \omega \in \mathfrak v^* \}.
\end{equation}
One can think of~${\mathfrak v}^*$ as a sheet above $\lambda=0$.
The case of $\lambda=0$ and $\omega=0$ corresponds to the trivial representation of $G$.

\subsubsection{The Fourier transform}
\label{subsubec_FT}  In contrast with the Euclidean case, the Fourier transform  is defined on~$\widehat G$ and is valued  in   the space of  bounded operators
on~$L^2( \mathbb R^d)$. More precisely, the Fourier transform of a function~$f$ in~$L^1(G)$ is  defined  as follows: 
for any~$\lambda\in{\mathfrak z}^*$, $\lambda\not=0$,
 $$
 \widehat f(\lambda):=
{\mathcal F}(f)(\lambda):=
\int_G f(x)\left( \pi^{\lambda}_{x }\right)^* \, dx \, .
$$
Note that    for any~$\lam\in{\mathfrak z}^*$, $\lambda\not=0$, we have $\left( \pi^{\lambda}_{x }\right)^* =\pi^{\lambda}_{x^{-1} }$ and the map~$\pi^{\lambda}_{x}$
 is a group homomorphism from~$G$ into the group~$U (L^2( \mathfrak p_\lambda))$  of unitary operators
of~$L^2( \mathfrak p_\lambda)$ (often identified with 
$L^2( \mathbb R^d)$).
Therefore, the function ~$f$ in~$L^1(G)$  has a Fourier transform~$\left({\mathcal F}(f)(\lambda)\right)_{\lambda}$ which is a bounded family of bounded operators on~$L^2( \mathfrak p_\lambda)$
 with uniform bound:
\begin{equation}
\label{eq_FfnormL1}
\|\mathcal Ff (\lambda) \|_{{\mathcal L}(L^2(\mathfrak p_\lambda))}
\leq
\int_G |f(x)|\|(\pi^\lambda_x)^* \|_{{\mathcal L}(L^2(G))} dx
=
\|f\|_{L^1(G)}.
\end{equation}
since the unitarity of $\pi^\lambda$ implies $\|(\pi^\lambda_x)^* \|_{{\mathcal L}(L^2(G))}=1$.

\medskip

The Fourier transform can be extended to an isometry from~$L^2(G)$ onto the Hilbert
space of families~$ A  = \{ A (\lam ) \}_{(\lambda) \in{\mathfrak z}^*\setminus \{0\}}$
 of operators on~$L^2( \mathfrak p_\lambda)$ which are
Hilbert-Schmidt for almost every~$\lambda\in{\mathfrak z}^*\setminus \{0\}$, with~$\|A (\lam)\|_{HS (L^2( \mathfrak p_\lambda))}$ measurable and with norm
\[ \|A\| := \left( \int_{\mathfrak z^* \setminus\{0\}}
\|A (\lam )\|_{HS (L^2( \mathfrak p_\lambda))}^2 |\lambda|^d \, d\lam
\right)^{\frac{1}{2}}<\infty  \, .\]
  We have the following Fourier-Plancherel formula: 
 \begin{equation}
\label{Plancherelformula} \int_G  |f(x)|^2  \, dx
=  c_0 \, \int_{\mathfrak z^* \setminus\{0\}} \|{\mathcal F}(f)(\lambda)\|_{HS(L^2( \mathfrak p_\lambda))}^2 |\lambda|^d  \,  d\lambda   \,, 
\end{equation}
where $c_0>0$ is a computable constant.
This yields  an inversion formula for any Schwartz function $ f \in {\mathcal S}(G)$ and $x\in G$:
\begin{equation}
\label{inversionformula} f(x)
= c_0 \, \int_{\mathfrak z^* \setminus\{0\}} {\rm{Tr}} \, \Big(\pi^{\lambda}_{x} {\mathcal F}f(\lambda)  \Big)\, |\lambda|^d\,d\lambda \,,
\end{equation}
where ${\rm Tr}$ denotes the trace of operators in ${\mathcal L}(L^2({\mathfrak p}_\lambda))$.The inversion  formula makes sense since for $f \in {\mathcal S}(G)$, the operators ${\mathcal F}f(\lambda)$, $\lambda\in \mathfrak z^*\setminus\{0\}$, are trace-class and $\int_{\mathfrak z^* \setminus\{0\}} {\rm{tr}} \, \Big| {\mathcal F}f(\lambda)  \Big|\, |\lambda|^d\,d\lambda$ is finite.

\medskip

Usually, the Fourier transform of a locally compact group $G$ would be defined on $\widehat G$, the set of unitary irreducible representations of $G$ modulo equivalence, via
$$
\widehat f(\pi)=
\mathcal F (f)(\pi) = \int_G f(x) \pi(x)^* dx, 
$$
for a representation $\pi$ of $G$, and then considering the unitary equivalence we obtain a measurable field of operators $\mathcal F (f)(\pi)$, $\pi\in \widehat G$.
Here, the Plancherel measure is supported in the subset $\{\mbox{class of} \ \pi^\lambda \ : \ \lambda \in \mathfrak z^* \setminus\{0\}\}$ 
 of $\widehat G$ (see \eqref{eq_widehatG})
 since it is $c_0|\lambda|^d d\lambda$.
 This allows us to identify $\widehat G$ and 
 ${\mathfrak z}^*\setminus\{0\}$ when considering measurable objects up to null sets for the Plancherel measure. 
 However, our semiclassical analysis will lead us to consider objects which are also supported in the other part of $\widehat G$.
 For this reason, we also set for $\omega\in \mathfrak v^*$ and $f\in L^1(G)$:
\begin{align*}
\widehat f(0,\omega)=
 \mathcal F(f) (0,\omega) 
 &:= \int_G f(x) (\pi^{(0,\omega)}_x)^* dx
\\
&=\int_{\mathfrak v \times \mathfrak z} 
f({\rm Exp}(V+Z) ) e^{-i\omega(V)} dV dZ.	
\end{align*}

\medskip

The Fourier transform  sends the    convolution, whose definition is recalled in~(\ref{convolutiondef}),
to  composition in the following way:
\begin{equation}\label{fourconv}
 {\mathcal F}( f \star g )( \lam ) 
 = {\mathcal F} (g)( \lam)\
 {\mathcal F}(f) ( \lam) \, .
 \end{equation}
Other conventions for the convolution or in having (or not) the adjoint $\left( \pi^{\lambda}_{x }\right)^*$ in the formula for the Fourier transform would lead to obtaining 
${\mathcal F} (f)(\lam)\ {\mathcal F}(g) (\lam)$. 
However, we made the consistent choices of privileging left objects (e.g. in our choice of convolution and identification of the Lie algebra $\mathfrak g$ with the Lie algebra of left invariant vector fields). 
The choice of considering the adjoint in the Fourier transform is natural from an analytical viewpoint. A first reason is to extend the case of the Euclidean Fourier transform on the abelian group $\R^n$ where the formula usually contains $e^{-ix\xi}$.
A deeper reason is that given our choice of convolution, it is usual to consider right convolution operators since they are invariant  under (i.e. commute with) left translations. 
For such an operator $T$ with, say, integrable convolution kernel $\kappa\in L^1(G)$, this means that $Tf=f*\kappa$ and we have $\mathcal F (Tf) = \mathcal F (\kappa) \ \mathcal F f$ by \eqref{fourconv}; in other words, $T$ is a Fourier multipliers with Fourier symbol $\mathcal F (\kappa)$ acting on the left (and not the right) of $\mathcal F f$, and this setting seems quite natural.
Note that the formula $\mathcal F (Tf) = \mathcal F (\kappa) \ \mathcal F f$ implies
\begin{equation}
\label{eq_cq_plancherel}
	\|T\|_{{\mathcal L}(L^2(G))} = \sup_{\lambda\in \widehat G} 
\|\mathcal F\kappa (\lambda) \|_{{\mathcal L}(L^2(\mathfrak p_\lambda))},
\end{equation}
where the supremum here is in fact the essential supremum with respect to the Plancherel measure, justifying the identification of the element in $\widehat G$ with $\lambda\in \mathfrak z^*$.

\medskip 

With our conventions, we also have
\begin{equation}\label{F(V)}
\mathcal F(Xf)(\pi) = \pi(X) \mathcal F(f) (\pi)
\end{equation}
where $\pi(X)$ is the infinitesimal representation of $\pi$ at $X\in \mathfrak g$, 
\begin{equation}
\label{eq_pi(X)}
\mbox{i.e.}\ 
\pi(X) = \frac{d}{dt} \pi ({\rm Exp}(tX))|_{t=0}; 	
\end{equation}
 (the class of) $\pi$ is equal to (the class of) $\pi^\lambda$ or $\pi^{(0,\omega)}$ identified with $\lambda$ or $\omega$ respectively.
For instance, we have for $Z\in \mathfrak z$ identified with a left-invariant vector field
\begin{equation}
\label{def:Z}
 {\mathcal F} (Zf)(\lambda)=i \lambda(Z)  {\mathcal F}(f)(\lambda),\quad 
 \mbox{or in other words}\quad  \widehat Z (\lambda)  = i\lambda (Z).
 \end{equation}
The infinitesimal representation of $\pi$ extends to the universal enveloping Lie algebra of $\mathfrak g$ that we identify with the left invariant differential operators on $G$.
Then for such a differential operator $T$ we have 
$\mathcal F(Tf)(\pi) = \pi(T) \mathcal F(f) (\pi)$
and we may write $\pi(T)={\mathcal F}(T)$. 
For instance, if as before ${\mathcal X}^{\alpha}$ denotes a product of $|\alpha|$  left invariant vector fields forming a basis of~${\mathfrak g}$, then 
$$\mathcal F({\mathcal X}^\alpha f)(\pi) = \pi({\mathcal X})^\alpha \mathcal F(f) (\pi)\;\;{\rm and}\;\;
\mathcal F({\mathcal X}^\alpha ) = \mathcal F({\mathcal X})^\alpha.$$ 
 
\subsection{The sublaplacian} \label{freq}
The sublaplacian on $G$ is defined by
$$
\Delta_{G}:= \sum_{j=1}^{2d} V_j^2, 
$$
where the $V_j$'s form  an orthonormal basis of $\mathfrak v$.
One checks easily that $\Delta_G$ is a differential operator which is left invariant and homogeneous of degree $2$.

\medskip 

The definition of $\Delta_G$ is independent of the chosen orthonormal basis for $\mathfrak v$ - although it depends on the scalar product that we have fixed at the very beginning on ${\mathfrak v}$. 
In particular, choosing the basis fixed in Section \ref{Fourier}  for any $\lambda\in \mathfrak z^* \setminus\{0\}$
we have
$ \Delta_{G}= \sum_{j = 1}^{d} (P_j^2 + Q_j^2).$
This allows us to compute its infinitesimal representation at $\pi^\lambda$:
\begin{equation}\label{def:H}
{\mathcal F} (-\Delta_G)(\lambda) = H(\lambda),
\end{equation}
where~$ H(\lambda)$ is the  diagonal  operator  defined on 
$L^2(\R^d)$ by
\begin{equation}\label{def:H2}
H(\lambda )=|\lambda| \sum_{1\leq j\leq d} \left( -\partial_{\xi_j}^2+\xi_j^2\right).
\end{equation}
The eigenfunctions of $H(\lambda)$  express in terms of the basis of Hermite functions 
 $(h_n)_{n\in\N}$,  normalized in~$L^2(\R) $ and satisfying
for all real numbers~$\xi $:
$$
-h''_n(\xi)+\xi^2 h_n(\xi)= (2n+1) h_n(\xi) \, .
$$  
Indeed, for each multi-index~$\alpha \in {\mathbb N}^d$, 
the function $h_{\alpha}$ 
defined  by
$$
 h_{\alpha} (\xi)  :=\prod_{j=1}^d h_{\alpha_j}(\xi_j),  \;\;\xi = (\xi_1,\dots, \xi_d) \in \R^d,
$$
is an eigenfunction of $H(\lambda)$, that is, 
$H(\lambda) h_{\alpha} =   \zeta(\alpha,\lambda) h_{\alpha}$,
for the eigenvalue 
$$
\zeta(\alpha,\lambda):=|\lambda|  \sum_{j =1}^d (2 \alpha_j + 1) = |\lambda| (2|\alpha|+d)  \, , \quad \alpha \in {\mathbb N}^d.
$$
These eigenvalues describe the entire spectrum of $H(\lambda)$.

\medskip 

We compute easily the  infinitesimal representation (or Fourier transform) of $\Delta_G$ on the rest of $\widehat G$: at $\pi^{(0,\omega)}$, it is the number
$$
{\mathcal F} (-\Delta_G)(0,\omega) = |\omega|^2.
$$
The eigenvalues for the Fourier transform of the sub-laplacian allow us to link the infinite dimensional representations $\pi^\lambda$ and the dimension~$1$ ones $\pi^{0,\omega}$:

\begin{lemma}\label{lem:dimfinie}
Let $(\lambda_n)_{n\in \mathbb N}$ and $(\alpha_n)_{n\in \mathbb N}$ be sequences in $\mathfrak z^*\setminus \{0\}$ and $\mathbb N$ respectively such that
$$
\lambda_n\Tend {n}{+\infty}  0
\quad \mbox{and}\quad \alpha_n \Tend {n}{+\infty} +\infty
\quad\mbox{while}\quad
|\lambda_n| (2|\alpha_n|+d)\Tend {n}{+\infty} \mu^2
$$
for some given $\mu\in\R^*_+$.
Then for any function $f\in L^1(G)$ which is radial
(i.e. $f({\rm Exp}(V+Z)=\tilde f (|V|,Z)$ for some mesurable function $\tilde f$), we have
$$
\left(\widehat f(\lambda_n) h_{\alpha_n},h_{\alpha_n}\right)
\Tend{n}{+\infty} 
\int_{\varsigma\in{\bf S}^{2d-1}}\widehat f(0,\mu \varsigma )
d\varsigma
=
\int_{\R^{2d} \times \R^p \times {\bf S}^{2d-1}} 
f({\rm Exp}(V+Z) ) e^{-i \mu \varsigma V} dV dZ d\varsigma.
$$
\end{lemma}

\begin{proof}
The following facts are well known, see \cite{koranyi,damek+ricci,faraud,hiero}. 
The function defined by 
$$
\Psi_{\lambda,|\alpha|}:x={\rm Exp}(V+Z)\longmapsto 
\int_{SO(\mathfrak v)}((\pi^\lambda_{{\rm Exp}(kV+Z)})^* h_\alpha,h_\alpha) dk
$$
is smooth and bounded by 1 (its value at ${\rm Exp}( 0)$) on $G$. 
 It is equal to 
$$
\Psi_{\lambda,|\alpha|}({\rm Exp}(V+Z))
= 
e^{-i\lambda(Z)}
 \mathcal L_{|\alpha|}^{d-1} \left(\frac {|\lambda|}2 |V|^2\right),
$$
where $\mathcal L_k^\delta$ is the normalised Laguerre function of type $\delta=d-1$ and degree $k\in \mathbb N$;
this means that 
$\mathcal L_k^\delta (t) = \frac{ L_k (t)}{L_k(0)} e^{-\frac{t}2}$, $t\in \R$, 
where $L_k (t) = t ^{-\delta} e^t \frac 1 {k!} (\frac{d}{dt})^k (e^{-t} t^{k+\delta})$ is the Laguerre polynomial of type $\delta>-1$ and degree $k\in \mathbb N$.
The function defined on $G$ by
$$
\Psi_{|\omega|}:x={\rm Exp}(V+Z)\longmapsto 
\int_{SO(\mathfrak v)}(\pi^{(0,\omega)}_{{\rm Exp}(kV+Z)})^*  dk.
$$
is smooth and bounded by 1 (its value at ${\rm Exp}( 0)$). We have
$$
\Psi_{|\omega|}({\rm Exp}(V+Z))=
\int_{\varsigma\in{\bf S}^{2d-1}}e^{-i |\omega|\varsigma V} d\varsigma ,
$$
and this expression is a known function of $|\omega||V|$ (this known function is sometimes called the reduced Bessel function). 
Properties of the Laguerre and Bessel functions implies \cite{faraud}  the convergence 
$\Psi_{\lambda_n,\alpha_n} 
\Tend{n}{+\infty} 
\Psi_\mu$ 
on any compact of $G$
when  the sequences $(\lambda_n)_{n\in \mathbb N}$ and $(\alpha_n)_{n\in \mathbb N}$ are as in the statement. 
The result follows easily from this and the calculations
$$
\left(\widehat f(\lambda_n) h_{\alpha_n},h_{\alpha_n}\right)
=\int_G f(x) \Psi_{\lambda_n,|\alpha_n|}(x) dx,
\qquad
\int_{\varsigma\in{\bf S}^{2d-1}}\widehat f(0, \mu \varsigma )
d\varsigma
=\int_G f(x) \Psi_{|\omega|}(x)dx.
$$
\end{proof}

\begin{remark}\label{rem:magenta}
Note that the proof above shows that we also have for any $f\in{\mathcal S}(G)$,
$$\int_{SO(\mathfrak v)} \left(\widehat f_k(\lambda_n)h_{\alpha_n},h_{\alpha_n}\right) dk 
\Tend{n}{+\infty} 
\int_{\varsigma\in{\bf S}^{2d-1}}\widehat f(\mu \varsigma )
d\varsigma,
$$
where $f_k({\rm Exp}(V+Z))=f({\rm Exp}(kV+Z))$ for all $k\in SO(\mathfrak v)$. Here $dk$ the normalised (probability) Haar measure of the compact group $ SO(\mathfrak v)$.
\end{remark}

The function $\Psi_{\lambda,|\alpha|}$ and $\Psi_{|\omega|}$ are called the bounded spherical functions of $G$ \cite{koranyi,damek+ricci,faraud,hiero}. 
They can be characterised in the following way:
the associated linear maps $f\mapsto \int_G f(x) \Psi(x)dx$ for 
$\Psi=\Psi_{\lambda,|\alpha|}$, 
$\lambda\in \mathfrak z^* \setminus \{0\}$, $\alpha\in \mathbb N^d$,  
and $\Psi_{|\omega|}$ $\omega\in \mathfrak v^*$, 
are in fact the characters of the commutative convolution algebra of radial integrable functions on $G$.
Another viewpoint is that  these functions are general eigenfunctions of $\Delta_G$ and $Z_1,\ldots,Z_p$:
$$
\Delta_G \Psi_{\lambda,|\alpha|}  =- \zeta(\alpha,\lambda) \Psi_{\lambda,|\alpha|}, \quad 
\Delta_G \Psi_{|\omega|} = -|\omega|^2 \Psi_{|\omega|},
\quad 
Z_j \Psi_{\lambda,|\alpha|}  =i\lambda_j \Psi_{\lambda,|\alpha|} 
\quad\mbox{and}\quad
Z_j \Psi_{|\omega|} = 0.
$$
Moreover, they yield the joint spectral decomposition of $\Delta_G$
and $Z_1,\ldots, Z_p$.
The joint spectrum in the case of $p=1$, that is, in the case of $G=\mathbb H_d$ may be drawn in what is often called the Heisenberg fan.

\medskip

\begin{figure}
\includegraphics[height=0.33\textheight]{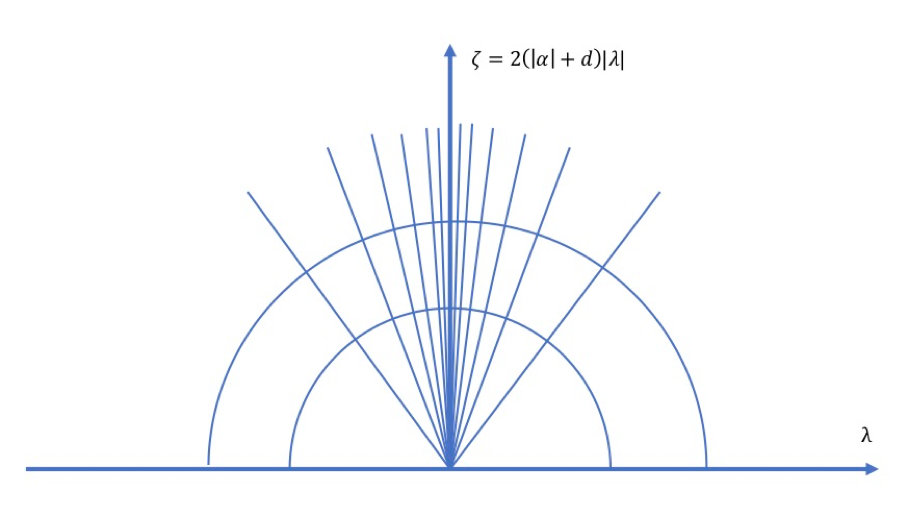}
\caption{Heisenberg's fan }
\label{fan}
\end{figure}

\medskip

Putting~$\lambda$ on the horizontal axis  and $\zeta$ on the vertical axis, the set of half lines $\zeta=|\lambda| (2|\alpha|+d)$ when~$\alpha$ describes $\N^d$ concentrates on the vertical line as pictured in Figure~\ref{fan}.
This figure is known as the Heisenberg fan. 
It can be viewed as the spectrum of the convolution algebra of radial and integrable functions on $G$ as well as the joint spectrum of $-\Delta_G, iZ_1,\ldots,Z_p$, see for instance \cite{ADR,FRY1}.


\section{Semi-classical pseudodifferential operators}

As mentioned above, the Fourier transform being operator-valued, so will be the symbols of semi-classical pseudodifferential operators that we shall now define. 

\subsection{The algebra of symbols $\mathcal A_0$}

We denote by $\mathcal A_0$ the space of symbols 
$\sigma = \{\sigma(x,\pi) : (x,\pi)\in G\times \widehat G\}$ of the form 
$$
\sigma(x,\lambda)=\mathcal F \kappa_x (\lambda) = \int_G \kappa_x(y) (\pi^\lambda_x)^* dx, 
$$
where $x\mapsto \kappa_x(y)$ is a smooth and compactly supported function from $G$ to $\mathcal S(G)$.
Being compactly supported means that $\kappa_x(y)=0$ for $x$ outside a compact of $G$ and any $y\in G$.

\begin{remark}\label{rem:sigmamu} In
 the case of representations of finite dimension,  we  distinguish between all the finite dimensional representations by replacing $\lambda=0$ by the parameters $(0,\omega)$, $\omega\in \mathfrak v^*$. Besides, 
 the  operator $\mathcal F \kappa_x(0,\omega)= \sigma(x,(0,\omega))$ reduces to a complex number because ${\mathcal H}_0=\C$. 
\end{remark}

As the Fourier transform is injective, it yields a one-to-one correspondence between the symbol~$\sigma$ and the function $\kappa$:
we have $\sigma(x,\pi)=\mathcal F \kappa_x(\lambda)$ and conversely the Fourier inversion formula \eqref{inversionformula}
yields 
$$
\forall x,z\in G,\;\; \kappa_x(z)= c_0\int_{\widehat G} {\rm Tr} \left( \pi^{\lambda}_ {z} \sigma(x,\lambda)\right) |\lambda|^d d\lambda.$$
The set $\mathcal A_0$ is an algebra for the composition of symbols since 
if $\sigma_1(x,\lambda)=\mathcal F \kappa_{1,x} (\lambda)$
and $\sigma_2(x,\lambda)=\mathcal F \kappa_{2,x} (\lambda)$
are in $\mathcal A_0$, 
then so is $\sigma_1(x,\lambda)\sigma_2(x,\lambda)
=\mathcal F (\kappa_{2,x}*\kappa_{1,x}) (\lambda)$
by \eqref{fourconv}.

\medskip 

 We can define two norms on $\mathcal A_0$:
one expressed directly on the symbol $\sigma$
 \begin{equation}
 \label{eq_normA}	
  \| \sigma\|_{\mathcal A}:= 
  \sup_{(x,\lambda)\in G\times \widehat G} 
  \|\sigma(x,\lambda)\|_{\mathcal L (L^2(\mathfrak p_\lambda))},
\end{equation}
and the other expressed on the associated function $\kappa_x =\mathcal F^{-1} \sigma(x,\cdot)$
 \begin{equation}
 \label{eq_normA0}	
  \| \sigma\|_{\mathcal A_0}:= \int_G \sup_{x\in G} |\kappa_x (z)|dz.
\end{equation}
From an algebraic viewpoint, they are both submultiplicative on $\mathcal A_0$:
$$
\| \sigma_1\sigma_2\|_{\mathcal A_0}
=
\int_G \sup_{x\in G} |\kappa_{2,x}*\kappa_{1,x} (z)|dz
\leq
\int_G (\sup_{x_2\in G} |\kappa_{2,x_2}|)*(\sup_{x_1\in G}|\kappa_{1,x} |)(z) dz
\leq 
\| \sigma_2\|_{\mathcal A_0}\| \sigma_1\|_{\mathcal A_0},
$$
and 
$$
\| \sigma_1\sigma_2\|_{\mathcal A}
\leq \!\!\!\!\!\!
 \sup_{(x_1,\lambda_1)\in G\times \widehat G} \!\!\!\!\!\!
  \|\sigma_1(x_1,\lambda_1)\|_{\mathcal L (L^2(\mathfrak p_{\lambda_1}))}
 \!\!\!\!\!\!\sup_{(x_2,\lambda_2)\in G\times \widehat G} \!\!\!\!\!\!
  \|\sigma_2(x_2,\lambda_2)\sigma_1(x,\lambda)\|_{\mathcal L (L^2(\mathfrak p_{\lambda_2}))}
  =
\| \sigma_1\|_{\mathcal A}\| \sigma_2\|_{\mathcal A}.
$$
By \eqref{eq_FfnormL1}, we have:
$$
\| \sigma\|_{\mathcal A}\leq 
  \sup_{x\in G} 
  \int_G |\kappa_x(z)| dz
\leq \| \sigma\|_{\mathcal A_0}.
$$
but one can show that the two norms are different. In fact, their counterparts on the abelian group~$\mathbb R^n$ are also different.

\medskip

Although we will not use the following, we observe that the algebra of symbol $\mathcal A_0$ may be defined as the space of smoothing symbols $S^{-\infty}(G)$ which are compactly supported in $x$; 
here the smoothing symbols are to be taken in the sense of \cite{FR}
(Section 5.2 for any graded nilpotent Lie group and Section 6.5 for the Heisenberg group).
In the context of an H-type group, 
a symbol $\sigma(x,\lambda)$ is smoothing when for all $\alpha,\beta \in \mathbb N^{2d+p}$ and for any $k_1,k_2\in \mathbb N$, there exists a constant $C_{\alpha,\beta,k_1,k_2}>0$ such that
\begin{equation}\label{estimate:seminorm}
\forall (x,\lambda)\in G\times (\mathfrak z^*\setminus\{0\})\qquad 
\left\| 
 H(\lambda)^{k_1}  \mathcal X^\beta_x 
 \Delta^\alpha \sigma(x,\lambda)
H(\lambda)^{k_2}
\right\|_{{\mathcal L}(L^2(\R^d))}
\leq C_{\alpha,\beta,k_1,k_2}.
\end{equation}
Here we assume that a basis $X_1,\ldots,X_{2d+p}$ of $\mathfrak g$ has been fixed. As before $\mathcal X^\beta_x$ is the (ordered) product of $|\beta|$ of these left invariant vector fields.
The difference operators $\Delta^\alpha$ are related with the coordinate functions 
 $\phi_1,\ldots, \phi_{2d+p}$ associated with the basis $X_1,\ldots,X_{2d+p}$, i.e. $x = {\rm Exp}(\phi_1(x) X_1 +\ldots + \phi_{2d+p}(x) X_{2d+p})$; denote   by $\Delta_{\phi_1},\ldots, \Delta_{\phi_{2d+p}}$ the corresponding difference operators,
$$
(\Delta_{\phi_j} \sigma)(x,\lambda)= \mathcal F(\phi_j \kappa_x)
\quad\mbox{if} \ \sigma(x,\lambda)=\mathcal F \kappa_x(\lambda).
$$
Then $\Delta^\alpha$ is the product of $|\alpha|$ of such difference operators. Note that two difference operators $\Delta_{\phi_1}$ and $\Delta_{\phi_2}$ of that form will commute since 
$\Delta_{\phi_1}\Delta_{\phi_2}=\Delta_{\phi_1\phi_2}$.  Then, the estimate~(\ref{estimate:seminorm}) above yields semi-norms on $S^{-\infty}(G)$ 
and induces there a topology by inductive limit.

\medskip

More generally, one can always define a difference operator associated with a function $\phi$ of polynomial growth (i.e. $|\phi(x)|\leq (1+|x|)^N$ for some $N\in \mathbb N$) as the operator defined on $\mathcal A_0$ by
\begin{equation}\label{difference}
(\Delta_{\phi} \sigma)(x,\lambda)= \mathcal F(\phi \kappa_x)
\quad\mbox{if} \ \sigma(x,\lambda)=\mathcal F \kappa_x(\lambda).
\end{equation}

\subsection{Semi-classical pseudodifferential operators}

Let $\eps>0$ be a small parameter, the {\it semi-classical parameter} that we shall use to weight the oscillations of the functions that we shall consider. Following~\cite{BFG1} and~\cite{FR}, we quantify the symbols that we have introduced previously by setting  
\begin{equation}
	\label{eq_quantization}
	{\rm Op}_\eps(\sigma) f(x)= c_0 \int_{\widehat G} {\rm Tr} \left(  \pi^{\lambda}_{x} \sigma(x,\eps^2\lambda) {\mathcal F} f(\lambda)  \right)|\lambda|^d \,d\lambda,\;\; f\in{\mathcal S}(G).
\end{equation}
The kernel of the operator ${\rm Op}_\eps(a)$ is the function 
$$
G\times G\ni (x,y)\mapsto \kappa^\eps_x(y^{-1} x)
$$ where $\kappa^\eps_x(z)=\eps^{-Q} \kappa_x\left(\delta_{\eps^{-1}} z\right)$ 
and $\kappa_x $ is such that $\mathcal F(\kappa_x)(\lambda)=\sigma(x,\lambda)$.
For this reason, we call $\kappa_x$ the {\it convolution kernel} of $\sigma$.

\medskip 

The norm $\|\cdot \|_{\mathcal A_0}$ defined in \eqref{eq_normA0}
allows us to bound the 
action of the symbols in ${\mathcal A_0}$ (and more generally in $S^{-\infty}(G)$) on $L^2(G)$:

\begin{proposition}
\label{prop_L2bdd}
Let $\sigma\in {\mathcal A}_0$, then ${\rm Op}_\eps(\sigma)$ is bounded in $L^2(G)$. Moreover, there exists a constant $C>0$ such that for all 
$$\forall \sigma\in {\mathcal A}_0,\;\;
\forall \eps>0,\;\;\| {\rm Op}_\eps(\sigma)\|_{{\mathcal L}(L^2(G))}\leq C \,\| \sigma\|_{\mathcal A_0}.$$
\end{proposition}

\begin{proof}
We observe that if $f\in \mathcal S(G)$ then 
$$
|{\rm Op}_\eps (\sigma)f (x)| =
|\int_G f(y) \kappa_x^{\eps} (y^{-1}x) dy|
\leq
\int_G |f(y)| \sup_{x_1\in G} |\kappa_{x_1}^{\eps} (y^{-1}x)|\ dy
= |f|* \sup_{x_1\in G} |\kappa_{x_1}^{\eps}(\cdot)|(x), 
$$
so the Young convolution inequality implies
$$
\|{\rm Op}_\eps (\sigma)f \|_{L^2(G)}
\leq 
\|f\|_{L^2(G)} \|\sup_{x_1\in G} |\kappa_{x_1}^{\eps}(\cdot)|\|_{L^1(G)}.
$$
We recognise this $L^1$-norm as $\|\sigma\|_{\mathcal A_0}$:
$$
\|\sup_{x_1\in G} |\kappa_{x_1}^{\eps}(\cdot)|\|_{L^1(G)}
=
\|\sup_{x_1\in G} |\kappa_{x_1}(\cdot)|\|_{L^1(G)}=\|\sigma\|_{\mathcal A_0}.
$$
\end{proof}

Notice that the boundedness of semi-classical pseudodifferential operators of symbols of the form $\tau(x,\lambda)= a(x)b(\lambda)$ is much easier and, since  ${\rm Op}_\eps(\tau)$ simply is the composition of the operator of multiplication by $a(x)$ and of the Fourier multiplier $b(\eps^2 \lambda)$, one has 
\begin{equation}\label{simplebound}
\| {\rm Op}_\eps(\tau)\|_{{\mathcal L}(L^2(G))}\leq \| \tau \|_{{\mathcal A}}.
\end{equation}
However, similarly to what happens in the Euclidean case, when the link between $x$ and $\lambda$ is more intricate inside the symbol $\sigma(x,\lambda)$, the estimate of ${\rm Op}_\eps(\sigma)$ requires a more elaborate norm, implying bound on the derivatives. The use of the norm $\|\cdot\|_{\mathcal A_0}$ is particularly convenient in this semi-classical frame. 
However, instead of $\|\cdot\|_{\mathcal A_0}$, we could have chosen a suitable semi-norm on $S^{-\infty}$ which guarantees the $L^2(G)$-boundedness of the corresponding operator, for instance a semi-norm on the class of symbol $S^0(G)$ from \cite[Theorem 5.4.17]{FR}.
 Then, with this seminorm instead of $\|\cdot\|_{\mathcal A_0}$, we would obtain an estimate similar to the one in Proposition~\ref{prop_L2bdd} for all $\eps \in (0,1)$.
 
 \medskip 

Note secondly that 
our quantization in \eqref{eq_quantization} is the analogue of the Kohn-Nirenberg quantization or left quantization on the abelian group $\mathbb R^n$ where it is more common to use the Weyl quantization, which is linked with the Wigner transform, 
see Section \ref{subsec_Wtransf}.

\medskip 
 
Let us conclude with a remark about the  generalisation of our result. 

\begin{remark}
The elements  above can be similarly developed for any graded Lie group $G$. Denoting by $\pi$ the elements of the set $\widehat G$ of its representations, it is possible to extend the dilations to $\widehat G$ and to define $\eps\cdot \pi$ (see Section~2.3 of~\cite{FF}) and for $\sigma\in S^{-\infty}(G)$, we set 
$${\rm Op}_\eps(\sigma)= c_0 \int_{\pi\in\widehat G} {\rm Tr}_{L^2({\mathcal H}_\pi)} \left(  \pi(x) \sigma(x,\eps\cdot \pi) {\mathcal F} f(\pi)  \right)d\mu(\pi),$$
where $d\mu(\pi)$ is the Plancherel measure on $\widehat G$. 
\end{remark}

In the next two paragraphs, we emphasize properties of the semi-classical pseudodifferential operators and revisit the difference operators which can be precisely computed in the context of H-type groups and appear in the symbolic calculus.

\subsection{Role of the diagonal in the kernel of a semi-classical pseudodifferential operator}

The diagonal plays an important role for the integral kernel of ${\rm Op}_\eps(\sigma)$ because that is where the singularities will lie. In fact, one can always assume that the function $(x,z)\mapsto \kappa_x(z)$ is compactly supported close to $z=0$ (note that $\kappa_x(z)$ is compactly supported in $x$ by definition):

\begin{proposition}\label{prop:diag}
Let $\chi\in{\mathcal C}^\infty_0(G) $ be identically equal to $1$ close to $0$. 
For every $\eps>0$ and $\sigma = \mathcal F(\kappa_x)\in \mathcal A_0$, 
the symbol defined via $\sigma_\eps(x,\lambda) = \mathcal F ( \kappa_x \chi(\delta_{\eps} \cdot )) (\lambda)$
is  in $\mathcal A_0$ and its  kernel 
$$(x,y)\mapsto \eps^{-Q} \kappa_x(\delta_{\eps^{-1}} (xy^{-1})\chi(x y^{-1})$$ 
is compactly supported close to the diagonal $x=y$. 
For all $N\in\N$, there exists a constant $C=C_{N,\sigma}>0$ such that 
$$
\forall \eps>0\qquad
\left\|  \sigma_\eps-\sigma\right\|_{\mathcal A_0}
\leq C {\eps}^{QN}.
$$
\end{proposition}

\begin{proof}
As the function $\chi$ is identically~$0$ close to $z=0$, for all $N\in\N$, there exists a bounded smooth function $\theta_N$  
such that 
$$\forall y\in G,\;\; \chi(y)-1=\theta(y)\| y\|^N,$$
where $\|y\| = (|V|^4 +|Z|^2)^{1/4}$ for $y={\rm Exp}(V+Z)\in G$.
Note that $\|\eps y\|=\eps \|y\|$.
We also set
$$b^{(N)}_x(z):= \| z\| ^N \kappa_x(z)$$
We compute easily
$$
\left\|  \sigma_\eps - \sigma\right\|_{\mathcal A_0}
=\eps^{NQ}   \int_G \sup_{x\in G} |b^{(N)}_x(z)  \theta(\delta_\epsilon z)|dz
\leq \eps^{NQ} \|\theta\|_{L^\infty} \int_G \sup_{x\in G} |b^{(N)}_x(z)  |dz.
$$
As this last integral is finite, this concludes the proof.
\end{proof}

\subsection{Difference operators of H-type groups}\label{sec:difference}

We now study the difference operators rapidly introduced in~(\ref{difference}).  
We fix orthonormal bases $V_1,\ldots V_{2d}$ of $\mathfrak v$ and $Z_1,\ldots,Z_p$ of $\mathfrak z$.
We denote by $v_1,\ldots,v_{2d}$ the  coordinates of a vector 
$V=v_1V_1+\ldots + v_{2d}V_{2d}$ in $\mathfrak v$
and by $z_1,\ldots,z_p$ the  coordinates of a vector 
$Z=z_1Z_1+\ldots+z_pZ_p$ in $\mathfrak z$. We associate with the functions $v_1,\cdots, v_{2d}$ the difference operators $\Delta_{v_1},\cdots,\Delta_{v_{2d}}$ as defined in~(\ref{difference}). In the case of H-type groups, as for the Heisenberg group (see~\cite[Section 6.3]{FR}), these  operators  simply express in terms of the parameter $\lambda\in {\mathfrak z}^*\setminus\{0\}$. 

\medskip

Recall that for each $\lambda\in \mathfrak z^*\setminus\{0\}$, 
we have already fixed in Section \ref{Fourier}
a $\lambda$-depending orthonormal basis $(P_1,\ldots,P_d,Q_1,\ldots,Q_d)$ with respect to which the representation $\pi^\lambda$ can be conveniently described. 
The corresponding difference operators considered not at every $\lambda' \in \mathfrak z^*\setminus\{0\}$ but just at $\lambda'=\lambda$ are also easy to compute:  
\begin{align}\label{def:Deltap}
\Delta_{p_j}\widehat f (\lambda) 
&= |\lambda|^{-1/2}
[\xi_j, \widehat f (\lambda)] 
= -i|\lambda|^{-1}[\pi^\lambda(Q_j) , \widehat f(\lambda)]
\quad\mbox{denoted by}\ \Delta_{p_j}^{\lambda} \widehat f
\\
\label{def:Deltaq}
\Delta_{q_j}\widehat f (\lambda) 
&= 
|\lambda|^{-1/2}
[i\partial_{\xi_j}, \widehat f (\lambda)] 
=	
i|\lambda|^{-1}[\pi^\lambda(P_j) , \widehat f(\lambda)]
\quad\mbox{denoted by}\ \Delta_{q_j}^{\lambda} \widehat f,
\end{align}
and   for $1\leq m\leq p$,
$$\Delta_{z_m}\widehat f (\lambda)  = i\partial_{\lambda_m} \widehat f (\lambda) +{1\over 2} {\lambda_m\over |\lambda|} \Delta_p\cdot \Delta_q \widehat f (\lambda) +{i\over 2}  {\lambda_m\over |\lambda|}\pi^\lambda(Q)\cdot \nabla_q \widehat f (\lambda) + {i\over 2}  {\lambda_m\over |\lambda|}\Delta_p\cdot \Pi^\lambda(P)\widehat f (\lambda) $$
that we denote by $ \Delta_{z_m}^{\lambda} \widehat f.$
Indeed,
the proof of \cite[Lemma 6.3.1]{FR} for the case of the Heisenberg group  extends naturally to H-type groups
because of the link between the representation $\pi^\lambda$ quotiented by its kernel and the Schr\"odinger representation of the Heisenberg group already mentioned in Section \ref{defirreducible}.
For the same reasons (see \cite[Examples 6.3.4 and 6.3.5]{FR}), we also obtain
\begin{equation}
\label{eq_DeltaH}	
\Delta_{p_j}^\lambda H(\lambda) = 2\pi^\lambda(P_j)=2|\lambda|\partial_{\xi_j}
\qquad\mbox{and}\qquad
\Delta_{q_j}^\lambda H(\lambda) = 2\pi^\lambda(Q_j)=2i|\lambda|{\xi_j},
\end{equation}
 where $H(\lambda)= -{\mathcal F} \Delta_G(\lambda)$ has been defined in~(\ref{def:H}) and~(\ref{def:H2}).
The notation $\Delta_{p_j}^{\lambda} $, $\Delta_{q_j}^{\lambda}$ and $\Delta_{z_m}^\lambda$ emphasises the fact that these are not difference operators.
However, they are helpful when expressing the difference operators $\Delta_{v_1},\cdots,\Delta_{v_n}$:

\begin{proposition}\label{prop:diffop}
Consider 
an orthogonal matrix $M^\lambda$ realising the change of basis from $(V_1,\ldots,V_{2d})$ to $(P_1,\ldots,P_d,Q_1,\ldots,Q_d)$.
Then, in vectoriel notation,
$$
\Delta_{v} \widehat f(\lambda)=
\begin{pmatrix}\Delta_{v_1} \widehat f(\lambda)
  \\ \vdots \\ \Delta_{v_{2d}} \widehat f(\lambda)
\end{pmatrix}
= (M^\lambda)^{-1}
\begin{pmatrix}
\Delta_{p}^\lambda \widehat f\\
\Delta_{q}^\lambda \widehat f
\end{pmatrix}
=|\lambda|^{-1/2}(M^\lambda)^{-1}
\begin{pmatrix} 
[\xi, \widehat f (\lambda)]  \\ 
[i\partial_{\xi}, \widehat f (\lambda)] 
\end{pmatrix}.
$$
\end{proposition}
\begin{proof}
The definition of $M^\lambda$ means that
if 
$$V=v_1V_1+\ldots + v_{2d}V_{2d} = p_1P_1+\ldots+ p_dP_d+ q_1Q_1+\ldots+q_dQ_d,$$ 
then
$$
\begin{pmatrix} p \\ q\end{pmatrix} = M^\lambda v,
\quad\mbox{where}\quad 
p=\begin{pmatrix}p_1\\ \vdots \\ p_d\end{pmatrix}, 
q=\begin{pmatrix}q_1\\ \vdots \\ q_d\end{pmatrix},
v=\begin{pmatrix}v_1\\ \vdots \\ v_{2d}\end{pmatrix}
$$
and the formula follows. 
\end{proof}

Note that due to the orthogonality of $M^\lambda$, we have
\begin{equation}
\label{eq_VDeltav}
V_1\Delta_{v_1} + \ldots + V_{2d} \Delta_{v_{2d}}=
V\cdot \Delta_v = (M^\lambda)^{-1} \begin{pmatrix}P\\Q\end{pmatrix}\cdot (M^\lambda)^{-1} \begin{pmatrix}\Delta^\lambda_p\\\Delta^\lambda_q\end{pmatrix}
= P\cdot \Delta^\lambda_p+Q\cdot\Delta^\lambda_q
\end{equation}
where the vector fields $V_j$, $P_j$ and $Q_j$ acts on the variable $x\in G$ whereas the difference operators $\Delta_{v_j}$ act on the variable $\lambda\in \widehat G$ of a symbol. 
In the same way, we have 
\begin{align}
\label{eq_Vpi(V)}	
V\cdot \pi^\lambda(V)
&=V_1 \pi^\lambda(V_1) + \ldots + V_{2d} \pi^\lambda (V_{2d})
\\
&= P_1 \pi^\lambda (P_1) + \ldots + P_d \pi^\lambda(P_d) + Q_1 \pi^\lambda(Q_1) + \ldots + Q_d\pi^\lambda(Q_d)\nonumber
\\
&= P\cdot \pi^\lambda(P)+ Q\cdot \pi^\lambda(Q),\nonumber
	\end{align}
acts on $x$ and $\lambda$ via the vector fields and the multiplication by $\pi^\lambda(V_j)$.

\subsection{Symbolic calculus}

These symbols enjoy a symbolic calculus.

\begin{proposition}\label{prop:symbcal}
Let $\sigma\in {\mathcal A}_0$. Then, in ${\mathcal L}(L^2(G))$, 
\begin{equation}\label{eq:adjoint}
{\rm Op}_\eps(\sigma)^*=  {\rm Op}_\eps (\sigma^*) -{\eps} \,{\rm Op}_\eps(P\cdot \Delta^\lambda_p \sigma^*+Q\cdot \Delta^\lambda_q \sigma^*)+O(\eps^2).
\end{equation}
Let $\sigma_1,\sigma_2 \in {\mathcal A}_0$. Then in ${\mathcal L}(L^2(G))$, 
\begin{equation}\label{eq:composition}
{\rm Op}_\eps(\sigma_1)\circ {\rm Op}_\eps(\sigma_2) = {\rm Op}_\eps(\sigma_1\,\sigma_2) -\eps\,  {\rm Op}_\eps\left( \Delta_p^\lambda \sigma_1 \cdot P\,\sigma_2+\Delta^\lambda_q \sigma_1 \cdot Q\, \sigma_2\right)+O(\eps^2),
\end{equation}
\end{proposition}

Note that in \eqref{eq:adjoint}, we can replace $P\cdot \Delta^\lambda_p \sigma^*+Q\cdot \Delta^\lambda_q \sigma^*$ with $V\cdot \Delta_v \sigma^*$, see \eqref{eq_VDeltav}.
Similarly, we can replace $\Delta_p^\lambda \sigma_1 \cdot P\,\sigma_2+\Delta^\lambda_q \sigma_1 \cdot Q\, \sigma_2$ with $\Delta_v \sigma_1\cdot V\sigma_2$ in \eqref{eq:composition}.

\medskip

A proof of Proposition \ref{prop:symbcal} is given in Appendix Section \ref{app_symbcal}. It follows the lines of the proofs of \cite[Section 5.5]{FR} with major simplification due to the semi-classical setting (as in the Euclidean case). The proof relies on Taylor formula, and, depending on the order to which this Taylor formula is pushed, one obtains more or less precise asymptotic expansions of the symbols. As in the Euclidean case, these expansions can be realized at any order. 

\begin{remark}
The difference operators $\Delta_{v_j}$  play the role of $-i\partial_{\xi_j}$ in the Euclidean setting. The reader will be able to check that the symbolic calculus formula are exactly the same as in the case of Kohn-Niremberg quantization in the Euclidean setting (see formula (3.17) and (3.18) in~\cite{gerard_X}).
\end{remark}


\section{Semi-classical measures}

Semi-classical measures are the adaptation of microlocal defect measures in a context where a  scale of oscillations is specified. 
This scale which is the semi-classical scale $\eps$ is prescribed by the sequence of family of functions to study or by the parameters of a given problem. 
Naturally, the section below highly relies on the work~\cite{FF}  which is devoted to microlocal defect measures on graded Lie groups. In particular, we use here a similar $C^*$-algebra approach.

\subsection{The $C^*$-algebra ${\mathcal A}$ of semi-classical symbols} 

We  introduce the algebra
 ${\mathcal A}$ which is  the closure of~${\mathcal A}_0$ for the norm
$\| \cdot\|_{\mathcal A}$ given in \eqref{eq_normA}. The algebra 
 ${\mathcal A}$ enjoys the properties of a $C^*$-algebra and one can identify its spectrum in the following way:

\begin{proposition}
\label{prop:C*A}
 The set ${\mathcal A}$ is a separable $C^*$-algebra of type~1.
It is not unital  but admits an approximation of identity. 
Besides, if $\lambda_0\in \widehat G$ and $x_0\in G$, then the mapping 
$$
\left\{\begin{array}{lll}
{\mathcal A}_0 &\longrightarrow& {\mathcal L}({\mathcal H}_{\lambda_0})\\
\sigma&\longmapsto & \sigma(x_0,\lambda_0)
\end{array}\right. \quad
$$
extends to a continuous mapping $\rho_{x_0,\pi_{\lambda_0}}:{\mathcal A}\to {\mathcal L}({\mathcal H}_{\lambda_0})$ which is an irreducible non-zero representation of ${\mathcal A}$. 
In fact, this is true for $\pi_0=\pi^{\lambda_0}$ and for $\pi_0=\pi^{0,\omega_0}$. Furthermore, the mapping
$$
R:\left\{\begin{array}{lll}
 G\times \widehat G &\longrightarrow&\widehat{\mathcal A} \\
 (x_0, \pi_0) &\longmapsto& \rho_{x_0,\pi_0}
 \end{array}\right.
 $$
is a homeomorphism which allows to identify  $\widehat {\mathcal A}$ with $G\times \widehat G$.
\end{proposition}

The proof follows the lines of \cite[Section 5]{FF}. It utilises the fact  that, by definition,  the $C^*$ algebra $C^*(G)$ of the group $G$ is the closure of $\mathcal F \mathcal S(G)$ for $\sup_{\lambda\in \widehat G} \|\cdot\|_{\mathcal L(\mathcal H_\lambda)}$ and that the spectrum of $C^*(G)$ is $\widehat G$. This implies readily that $\mathcal A$ may be identified with the $C^*$-algebra of continuous functions which vanish at infinity on $G$ and are valued on $C^*(G)$. It also implies that its spectrum is as described in Proposition \ref{prop:C*A}.

\medskip

We can also describe the states of the $C^*$-algebra $\mathcal A$. Still following \cite[Section 5]{FF}, we will need the following vocabulary:

\begin{definition}
\label{def_gammaGamma}
	Let $Z$ be a complete separable metric space, 
	and let $\xi\mapsto {\mathcal H}_\xi$ a measurable field of complex Hilbert spaces of $Z$.
\begin{itemize}
\item 
	The set 
	${\mathcal M}_1(Z,({\mathcal H}_\xi)_{\xi\in Z})$
	is the set of pairs $(\gamma,\Gamma)$ where $\gamma$ is a positive Radon measure on~$Z$ 
	and $\Gamma=\{\Gamma(\xi)\in {\mathcal L}({\mathcal H}_\xi):\xi \in Z\}$ is a measurable field of trace-class operators such that for all compact set $K\subset  Z$, 
	$$\int_{K}  {\rm Tr}\left| \Gamma(\xi)\right| d\gamma(\xi)<+\infty .$$
 \item 
	Two pairs $(\gamma,\Gamma)$ and $(\gamma',\Gamma')$ 
in ${\mathcal M}_1(Z,({\mathcal H}_\xi)_{\xi\in Z})$
are {equivalent} when there exists a measurable function $f:Z\to \mathbb C\setminus\{0\}$ such that 
$$d\gamma'(\xi) =f(\xi)  d\gamma(\xi)\;\;{\rm  and} \;\;\Gamma'(\xi)=\frac 1 {f(\xi)} \Gamma(\xi)$$ for $\gamma$-almost every $\xi\in Z$.
The equivalence class of $(\gamma,\Gamma)$ is denoted by $\Gamma d \gamma$.
\item 
A pair $(\gamma,\Gamma)$ 
in ${\mathcal M}_1(Z,({\mathcal H}_\xi)_{\xi\in Z})$
 is {positive} when 
$\Gamma(\xi)\geq 0$ for $\gamma$-almost all $\xi\in Z$.
In  this case, we may write $\Gamma d\gamma \geq 0$ or $(\gamma,\Gamma)\in{\mathcal M}^+_1(Z,({\mathcal H}_\xi)_{\xi\in Z})$.
\end{itemize}
\end{definition}

We will use the short-hands
$$
\mathcal M_1^+(G\times \widehat G) = \mathcal M_1^+(Z,({\mathcal H}_\xi)_{\xi\in Z})
\quad\mbox{when}\quad
Z=\{(x,\lambda) \in G\times \widehat G\}, 
\quad \mbox{and}\quad
\mathcal H_{x,\lambda}=\mathcal H_\lambda.
$$ 
We recall  that the Hilbert space ${\mathcal H}_\lambda$ is associated with the representation of $\lambda \in \widehat G$, 
that is, using the description in \eqref{eq_widehatG},
${\mathcal H}_\lambda=L^2(\mathfrak p_\lambda)$ if the representation corresponds to $\lambda \in \mathfrak z^*\setminus\{0\}$ and ${\mathcal H}_{(0,\omega)}=\C$ if $\lambda=0$ and the representation corresponds to $(0,\omega)$ with $\omega\in{\mathfrak v}^*$. 

\medskip

With this concept in mind, the states of the $C^*$-algebra ${\mathcal A}$ can be described as follows:
\begin{proposition}
\label{prop:states}
 If $\ell$ is a state of   ${\mathcal A}$, 
then there exists  $(\gamma,\Gamma)\in
{\mathcal M}_1^+(G\times \widehat G)$, unique up to its equivalence class, 
  satisfying
  \begin{equation}
\label{eq_prop_trGg}
\int_{G\times \widehat G} {\rm Tr} \left(\Gamma(x, \lambda)\right) d\gamma(x, \lambda) =1,
\end{equation}
and 
\begin{equation}
\label{eq_prop_elltrGg}
\forall \sigma\in{\mathcal A}\qquad
\ell(\sigma) = \int_{G\times \widehat G} {\rm Tr}\left(\sigma(x,\lambda) \Gamma(x, \lambda)\right) d\gamma(x, \lambda).
\end{equation}
Conversely, if a pair  
$(\gamma,\Gamma)\in {\mathcal M}_1^+(G\times \widehat G)$ satisfies \eqref{eq_prop_trGg}, then the linear form $\ell$ defined via \eqref{eq_prop_elltrGg} is a state of  ${\mathcal A}$. 
\end{proposition}

The proof of this proposition follows the lines of~\cite[Section 5]{FF} and  is given in Appendix Section~\ref{app_states}, for the convenience of the reader.

\subsection{Semi-classical measures}

We
 associate with a bounded family $(u^\eps)_{\eps>0}$ of $L^2(G)$  the quantities 
 \begin{equation}\label{def:leps}
\ell_\eps(\sigma)= \left({\rm Op}_\eps (\sigma) u^\eps,u^\eps\right)_{L^2(G)} ,\;\;\sigma\in \mathcal A_0,
\end{equation}
the limits of which are characterized by an element of ${\mathcal M}_1^+(G\times \widehat G)$.

\begin{theorem}\label{theo:mesures}
Let $(u^\eps)_{\eps>0}$ be a bounded family of $L^2(G)$.
There exist a sequence $(\eps_k)_{k\in \mathbb N}$ in $(0,+\infty)$ with  $\eps_k\Tend{k}{+\infty}0$
and a  pair $(\gamma,\Gamma)\in{\mathcal M}_1^+(G\times \widehat G)$
 such that we have
$$
\forall \sigma\in \mathcal A_0,\;\; 
\left({\rm Op}_{\eps_k} (\sigma) u^{\eps_k},u^{\eps_k}\right)_{L^2(G)}\Tend {k}{+\infty} \int_{G\times \widehat G} {\rm Tr}\left(\sigma(x,\lambda) \Gamma(x,\lambda)\right)d\gamma(x,\lambda).$$
Given the sequence $(\eps_k)_{k\in \mathbb N}$, 
the pair $(\gamma,\Gamma)\in{\mathcal M}_1^+(G\times \widehat G)$
is unique up to equivalence in ${\mathcal M}_1^+(G\times \widehat G)$
and satisfies 
$$\int_{G\times \widehat G} {\rm Tr}\left(\Gamma(x,\lambda)\right)d\gamma(x,\lambda) \leq \limsup_{\eps>0} \|u^\eps\|_{L^2(G)}^2.$$
\end{theorem}

 Any equivalence class $\Gamma d\gamma$  satisfying to Theorem~\ref{theo:mesures} for some subsequence $(\eps_k)_{k\in \mathbb N}$ is called a {\it semi-classical measure } of the family $(u^\eps)_{\eps>0}$.

   \begin{remark}
   \begin{enumerate}
\item   Note that this result can be generalized on any graded Lie group: there exist a sequence $(\eps_k)_{k\in \mathbb N}$ in $(0,+\infty)$ with  $\eps_k\Tend{k}{+\infty}0$
and a  pair $(\gamma,\Gamma)\in{\mathcal M}_1^+(G\times \widehat G)$
 such that we have 
 $$ \forall \sigma\in {\mathcal A}_0(G),\;\;\left({\rm Op}_{\eps_k} (\sigma) u^{\eps_k},u^{\eps_k}\right)_{L^2(G)}\Tend {k}{+\infty} \int_{G\times \widehat G} {\rm Tr}\left(\sigma(x,\pi) \Gamma(x,\pi)\right)d\gamma(x,\pi).$$
The pair $\Gamma d\gamma$ is called a semi-classical measure of the family $(u^\eps)$.
\item In the case of this article where $G$ is H-type, the special structure of $\widehat G$ implies that $\Gamma d\gamma$ consists of two pieces, one localized above $\lambda\in\mathfrak z^*\setminus\{0\}$ and another one which is scalar above $\mathfrak v^*$, see \eqref{eq_widehatG}.
\end{enumerate}
   \end{remark}

We now sketch the proof of Theorem~\ref{theo:mesures}.

\begin{proof}
By dividing $u^\eps$ by $\limsup_{\eps\rightarrow 0}  \| u^\eps\|_{L^2(G)}$ if necessary, we can assume that $$\limsup_{\eps\rightarrow 0}  \| u^\eps\|_{L^2(G)}=1.$$ 
Indeed, if $\limsup_{\eps\rightarrow 0}  \| u^\eps\|_{L^2G)}=0$, then the measure $\gamma=0$ answers our problem.
 We then consider the quantities 
$ \ell_\eps(\sigma)$ defined in~(\ref{def:leps}) and we observe the three following facts:
\begin{enumerate}
\item For any $\sigma\in {\mathcal A}_0$, the family $\ell_\eps(\sigma)$ is bounded and there exists a subsequence $(\eps_k(\sigma))_{ k\in \mathbb N}$ such that $\ell_{\eps_k(\sigma)}(\sigma)$ has a limit $\ell(\sigma)$
\item Using the separability of ${\mathcal A}_0$ and a diagonal extraction, one can find a sequence $(\eps_k)_{k\in \mathbb N}$ such that for all $\sigma\in {\mathcal A}_0$, the sequence $(\ell_{\eps_k}(\sigma))_{k\in \mathbb N}$ has a limit $\ell(\sigma)$ and the sequence $(\| u^{\eps_k}\|_{L^2(G)})_{k\in \mathbb N}$ converges to 1.
\item The map $\sigma \mapsto \ell(\sigma)$ constructed at point (2) is linear and satisfies $\ell(\sigma\sigma^*)\geq 0$ for all $\sigma\in  {\mathcal A}_0$.
\end{enumerate}
Observing that the set of symbols of the form $\tau(x,\lambda)=a(x)b(\lambda)$ is dense in ${\mathcal A}$, and that for these symbols, $|\ell(\tau)|\leq \|\tau \|_{\mathcal A}$ (see~(\ref{simplebound})), we can  extend the linear form~$\ell$ into a state of  ${\mathcal A}$, thus the existence of $\Gamma d\gamma$ by Proposition~\ref{prop:C*A}. 
\end{proof}

\subsection{Link with energy density and $\eps$-oscillation}

We want to link here the weak limits of the measure $|u^\eps(x)|^2 dx$ and the semi-classical measures of the family $u^\eps$. 
For this, we introduce the definition of an $\eps$-oscillating family of $L^2(G)$.

\begin{definition}
Let $(u^\eps)$ be a bounded family in $L^2(G)$.
 We shall say that $(u^\eps)$ is  \emph{$\eps$-oscillating} if 
$$
\limsup_{\eps\rightarrow 0} \left\|{\bf 1}_{-\eps^2\Delta_G>R} u^\eps
 \right\|_{L^2(G)}\Tend{R}{+\infty}0.$$
 \end{definition}

Let $\chi\in{\mathcal C}^\infty(\R)$ such that $0\leq \chi\leq 1$, $\chi=0$ on $]-\infty,1]$ and $\chi=1$ on $[2,+\infty[$. Equivalently, $u^\eps$ is $\eps$-oscillating if and only if 
$$\limsup_{\eps\rightarrow 0} \left\| \chi \left(-{\eps^2\over R} \Delta_G\right) u^\eps
 \right\|_{L^2(G)}\Tend{R}{+\infty}0.$$

\begin{proposition}\label{prop:sobcri}
If there exists $s>0$ and $C>0$ such that 
$$\forall \eps>0,\;\; \| (-\eps\Delta_G)^{s\over 2} \psi^\eps \|_{L^2(G)}
\leq C,$$
then $u^\eps$ is $\eps$-oscillating.
\end{proposition}

\begin{proof}
We use Plancherel formula and the facts that for $s>0$, 
$$  \chi \left(-{\eps^2\over R} \Delta_G\right) \leq {(-\eps^2\Delta_G)^{s\over 2}\over R^s}  \chi \left(-{\eps^2\over R} \Delta_G\right)\leq {(-\eps^2\Delta_G)^{s\over 2}\over R^s}.$$
\end{proof}

The interest of the notion of $\eps$-oscillation relies in the fact that it gives an indication of the size of the oscillations that have to be taken into account. It legitimates the use of semi-classical pseudodifferential operators and semi-classical measures. In particular, we have the following straightforward proposition.

\begin{proposition}\label{prop:eops}
Let $(u^\eps)$ be an $\eps$-oscillating family admitting a semi-classical measure $\Gamma d\gamma$ for the sequence $(\eps_k)_{k\in \mathbb N}$, then for all
  $\phi\in{\mathcal C}^\infty_0 (G)$
$${\lim_{k\rightarrow +\infty}  }\int_G \phi(x) |u^{\eps_k}(x)|^2 dx
=\int_{G\times \widehat G} \phi(x) {\rm Tr} \left(\Gamma(x,\lambda) \right) d\gamma(x,\lambda) .$$
\end{proposition}

\begin{proof} Let $\chi$ be as above. 
We write for any $R>0$$$
\int_G \phi(x) |u^{\eps_k}(x)|^2 dx =
I_0^{k,R} +I_1^{k,R},
$$
where
\begin{align*}
I_0^{k,R}
&:=	\int_G \phi(x) \ \chi\left(R^{-1} \eps_k^2\Delta_G \right) u^{\eps_k}(x) \ \overline{u^{\eps_k}(x)} dx, \\
I_1^{k,R}
&:=	\int_G \phi(x) \ (1-\chi)\left(R^{-1} \eps_k^2\Delta_G \right) u^{\eps_k}(x) \ \overline{u^{\eps_k}(x)} dx. 
\end{align*}
As $(u^\eps)$ is $\eps$-oscillating, 
$\lim_{R\to +\infty} \lim_{k\to +\infty}  I_0^{k,R}=0$.
For the other integral, it is known  that  $\phi(x) \ (1-\chi)\left( R^{-1} H(\lambda) \right) \in \mathcal A_0$, 
see for instance \cite[Corollary 3.8]{FF}, 
so Theorem \ref{theo:mesures} implies
\begin{align*}
\lim_{R\to +\infty}\lim_{k\to +\infty}  I_1^{k,R}
&= \lim_{R\to +\infty} \int_{G\times \widehat G} \phi(x) {\rm Tr}\left( (1-\chi)\left(R^{-1} H(\lambda) \right)
	\Gamma(x,\lambda)\right)  d\gamma(x,\lambda)
\\
&=\int_{G\times \widehat G} \phi(x) {\rm Tr}\left(	\Gamma(x,\lambda)\right)  d\gamma(x,\lambda).
\end{align*}
Combining the limits shows the statement.
\end{proof}

\subsection{Other quantizations and Wigner transform}
\label{subsec_Wtransf}

It is of course possible to quantize symbols $\sigma(x,\lambda)$ in another manner as it is in the Euclidean case and as discussed in the article~\cite{MantoiuRuzhansky} in a non-semi-classical setting.  One can associate with any $\tau\in(0,1)$ the $\tau$-quantization of the symbol $a\in{\mathcal A}_0$ by setting 
$${\rm Op}_\eps^\tau (a)  f(x)= c_0 \int_{ \widehat G} {\rm Tr} \left( \pi^\lambda_{x} \sigma(\delta_\tau x\,\delta_{1-\tau} y,\eps^2\lambda) {\mathcal F}(f)(\lambda)\right)|\lambda|^d d\lambda,\;\;
\forall f\in {\mathcal S}(G),\;\;\forall x\in G.$$
As in the Euclidean case, these quantizations are equivalent at leading order in the sense that we have the following proposition.

\begin{proposition}\label{prop:otherquant}
Let $\sigma\in{\mathcal A}_0$ and $\tau\in(0,1)$. There exists a constant $C>0$ such that 
$$\|{\rm Op}_\eps^\tau(\sigma)-{\rm Op}_\eps(\sigma)\|_{{\mathcal L}(L^2(G)}\leq C\eps.$$
\end{proposition}

The proof crucially uses the special of the kernel of the operator  ${\rm Op}_\eps^\tau(\sigma)$ which is the function 
$$G\times G \ni (x,y)\mapsto \eps^{-Q} \kappa_{\delta_\tau x \delta_{1-\tau} y} (\delta_{\eps^{-1}} (y^{-1} x))$$
with ${\mathcal F}(\lambda)(\kappa_x)=\sigma(x,\lambda)$. 
Set 
$$\kappa^{\eps,\tau}_x(z)= \kappa_{x\delta_{\eps(1-\tau)} z^{-1} }(z),\;\; z\in G,$$
then, 
the convolution kernel of ${\rm Op}_\eps^\tau(\sigma)$ is  the function
 $$G\times G \ni (x,z)\mapsto \eps^{-Q} \kappa^{\eps,\tau}_x(\eps^{-1} z).$$
The proof of Proposition \ref{prop:otherquant}
follows form the analysis of the difference between $\kappa^{\eps,\tau}_x(z)$ and $\kappa^{\tau}_x(z)$ thanks to a Taylor formula as in Proposition~\ref{prop:symbcal}.

\medskip

The quantization we have discussed until now corresponds to the choice of $\tau=1$. The analogue of what is the ``left'' quantization in the Euclidean case corresponds to $\tau=0$ and, as in the Euclidean case, one has for all $f\in {\mathcal S}(G)$,
$$\left({\rm Op}_\eps^1(\sigma(x,\lambda) )f,f\right) = \left(f,{\rm Op}_\eps^0(\sigma(x,\lambda)^*) f\right).$$
Similarly, 
the Weyl quantization should correspond to $\tau =1/2$. However, due to the non-commutativity of the law group,  the choice of $\tau=1/2$ does not induce a self-adjoint quantization, in the sense that we would have ${\rm Op}_\eps^\tau (\sigma)^*= {\rm Op}_\eps^\tau (\sigma^*)$. This constitutes a major  difference with the Euclidean case.

\medskip 

This implies that the definition of a ``natural'' Wigner distribution is not straightforward. Indeed, in the Euclidean case, the Wigner distribution of a function $f\in L^2(\R^d)$ is the self-adjoint distribution such that 
$$({\rm Op}^{1/2} (a) f,f)_{L^2(\R^d)} = \langle a,W^\eps _f\rangle.$$
Introduced in the 30s (see~\cite{wigner}), it plays the role of a generalized energy density defined in the phase space as its  marginal  gives the density in position or in impulsion:
$$|f(x)|^2 =\int_{\xi\in \R^d} W^\eps_f(x,\xi) d\xi,\;\;\eps^{-d/2}|\widehat f(\xi/\eps)|^2 =\int_{x\in \R^d} W^\eps_f(x,\xi) dx.$$ Even though it is not a positive distribution, it is self-adjoint, and it allows a unified treatment of space and Fourier variables which proves to be efficient in the analysis of a large range of problems.

\medskip

Let us now discuss what could be the generalization of Wigner distribution in  the case of H-type groups. 
We endow the space of  fields of operators defined on  $G\times \hat G $ with the inner product $\langle\cdot,\cdot\rangle$ defined by
$$ \langle A,B\rangle=c_0 \int_{G\times \widehat G} {\rm Tr}
\left(A(x,\lambda)B(x,\lambda)^*\right) |\lambda|^d d\lambda\, dx.$$
We then associate with a function $f\in{\mathcal S}(G)$ its {\it Wigner distribution} $W^\eps_f(x,\lambda)$ defined  by  
$$W^\eps_f(x,\lambda) ={c_0\over 2} \int_G \pi^{\lambda}_{w} \left(f(x \ \delta_\eps w^{-1} ) \overline f(x) +f(x ) \overline f(x\ \delta_\eps w) \right)dw,$$

\begin{enumerate}
\item The Wigner distribution of a function $f\in L^2(G)$ is a $\eps$-dependent  self-adjoint field of operators $W^\eps_f(x,\lambda)$.
\item  The Wigner distribution  $W^\eps_f(x,\lambda)$ is $L^1$ in the variable~$x$ and for $\gamma$-almost all $(x,\lambda)$, $W_f^\eps(x,\lambda)$ is a trace-class operator on ${\mathcal H}_\lambda$, besides 
$$\int _{G\times G} {\rm Tr} \left(W^\eps_f(x,\lambda) \right)|\lambda|^d d\lambda dx=\| f\|_{L^2(G)}^2.$$
\item For any $\sigma(x,\lambda)$, we have 
$$\langle W^\eps _f,\sigma\rangle = {1\over 2}\left( ({\rm Op}_\eps(\sigma )f,f)+(f,{\rm Op}_\eps(\sigma^*) f)\right) = {1\over 2}\left( ({\rm Op}_\eps^0(\sigma )f,f)+\left({\rm Op}_\eps^1(\sigma )f,f\right) \right).$$
\item The marginals of the Wigner distribution give the energy density in position and what we shall call the density in impulsion of $f$:
$$\displaylines{
\int_{\widehat G}  {\rm Tr} \, 
W^\eps_f(x,\lambda)|\lambda|^d d\lambda = |f(x)|^2 ,\cr
\int_G W^\eps_f (x,\lambda) dx = \eps^{-Q} \widehat f(\eps^{-2}\lambda) \widehat f(\eps^{-2}\lambda)^*.\cr}$$
\item If $u^\eps$ is a bounded family in $L^2(G)$, the semi-classical measures of $(u^\eps)$ are the weak limits of~$W^\eps_{u^\eps}$. 
\end{enumerate}


\section{Examples}\label{sec:examples}

In the Euclidean case, the basic examples are those families which present concentration on a point $x_0\in\R^n$ as the family $u^\eps$ described in~(\ref{suitesdel'intro}) or which oscillate along a vector $\xi_0\in\R^n$ as the sequence~$v^\eps$:
\begin{equation}\label{suitesdel'intro}
u_\eps(x)= \eps ^{-d/2} \Phi\left((x-x_0)/\eps\right)\quad\text{et}\quad v_\eps(x) = \Phi(x) {\rm e}^{i{x\cdot \xi_0/ \eps}}
\end{equation}
where $x_0,\xi_0\in\R^n$ and $\Phi\in{C}^\infty_0(\R^{n})$. These two families tend weakly to $0$ but not strongly because the point $x_0$ is an obstruction to the strong convergence for $u^\eps$ and similarly for $v^\eps$ for which the point~$\xi_0$ appears as a concentration point of the $\eps$-Fourier transform of $v^\eps$:
$$
(2\pi\eps)^{-d/2} \widehat v^\eps(\xi/\eps)=(2\pi\eps)^{-d/2} \widehat \Phi\left({\xi-\xi_0\over\eps}\right).$$
 The use of semi-classical measures allows one to unify the treatment of these two different behaviours, the  semi-classical measures of $(u^\eps)_{\eps>0}$ presenting a Dirac mass at $x_0$ and the one of $(v^\eps)_{\eps>0}$ a Dirac mass at $\xi_0$:
 $$\mu_{(u^\eps)} (x,\xi)= \delta(x-x_0)\otimes (2\pi)^{-d}|\widehat  \Phi(\xi)|^2  d\xi\;\;{\rm and}\;\; \mu_{(v^\eps)}(x,\xi)=|\Phi(x)|^2dx\otimes \delta(\xi-\xi_0).$$
 
 \medskip 
 
 We present in the next sections similar examples in the context of Lie groups with  a particular attention to sequences which concentrate on the  finite dimensional representations of $G$.

\subsection{Concentration}
In this section, we describe examples of bounded families of $L^2(\R^d)$ with semi-classical measures that concentrate on any point of $G$. 
Let $a\in{\mathcal S}(G)$ and $x_0\in G$, we set 
$$u^\eps(x) = \eps^{- {Q\over 2}} a(\delta_{\eps^{-1}} (xx_0^{-1})),\;\;x\in G.$$
Then $u^\eps$
is a strictly $\eps$-oscillating family since it satisfies the Sobolev criterium of Proposition~\ref{prop:sobcri}. 

\begin{proposition}
Any 
 semi-classical measure of the family $u^\eps$ is equivalent to the pair $(\Gamma,\gamma)$ with 
$$\Gamma(\lambda)=  \widehat a(\lambda)\widehat a(\lambda)^*,\;\; \gamma(x,\lambda)=c_0 \ \delta_{x_0}(x)\otimes \left(|\lambda|^dd\lambda\right).$$
\end{proposition}

\begin{proof}
We write the proof for $x_0=0$. 
We have
$$\left({\rm Op}_\eps(\sigma) u^\eps,u^\eps\right)_{L^2(G)} = c_0
\eps^{-Q}\int _{G\times G\times \widehat G} {\rm Tr} \left(\pi^\lambda_{ x}\sigma(x,\eps^2\lambda) (\pi^\lambda_y)^*\right) a(\delta_{\eps^{-1}} y)\overline{a(\delta_{\eps^{-1}} x)}|\lambda|^d dxdyd\lambda.$$
The change of variable $\delta_{\eps^{-1}} x\rightarrow x$, $\delta_{\eps^{-1}}  y\rightarrow y$ and $\eps^{2}\lambda \rightarrow \lambda$ and the fact that 
 $\pi^{\eps^{-2}\lambda}_{\delta_\eps x} = \pi^\lambda_x$ gives 
 $$\left({\rm Op}_\eps(\sigma) u^\eps,u^\eps\right)_{L^2(G)} = 
c_0\int _{G\times G\times \widehat G} {\rm Tr} \left(\pi^\lambda_{ x}\sigma(\delta_\eps x,\lambda) (\pi^\lambda_y)^*\right) a(y)\overline{a( x)}|\lambda|^d dxdyd\lambda,$$
whence the result in view of 
$$\int _{G\times G\times \widehat G} {\rm Tr} \left(\pi^\lambda_{ x}\sigma(0,\lambda) (\pi^\lambda_y)^*\right) a(y)\overline{a( x)}|\lambda|^d dxdyd\lambda
= \int _{ \widehat G} {\rm Tr} \left(\sigma(0,\lambda)  {\mathcal F}(a)(\lambda){\mathcal F}(a)(\lambda)^* \right) |\lambda|^d d\lambda.$$
\end{proof}

\subsection{Oscillations}
\label{subsec_osc}
In this section, we build examples of bounded families of $L^2(\R^d)$ with semi-classical measures that concentrate on some points of $\widehat G$, i.e. on the representations $\pi^\lambda_x$ for $\lambda\in {\mathfrak z}^*\setminus\{0\}$ and also for the finite dimensional representations $\pi_{0,\omega}$, $\omega\in{\mathfrak v}^*$. 

\medskip 
 
 For $x\in G$, we write $x={\rm exp}(V+Z)= x_{\mathfrak z} x_{\mathfrak v} = x_{\mathfrak v} x_{\mathfrak z} $ with $V\in{\mathfrak v}$, $Z\in{\mathfrak z}$,  $x_{\mathfrak z} ={\rm e}^Z\in G_{\mathfrak z}$ and $x_{\mathfrak v}={\rm e}^V\in G_{\mathfrak v}=G/ G_{\mathfrak z} $.
For $\lambda_\eps\in {\mathfrak z}^*$, we define the family
$$v_\eps (x)=|\lambda_\eps|^{d/2}e_\eps( x)a(x_{\mathfrak z}),\;\; e_\eps(x)=\left(\pi^{\lambda_\eps}_x \Phi_\eps, \Phi_\eps\right).$$
where 
$a\in{\mathcal S}(G_{\mathfrak z})$ is such that $\| a\|_{L^2(G_{\mathfrak z})}=1$ and 
 $\Phi_\eps \in {\mathcal S}(\R^d)$ with $\| \Phi_\eps\|_{L^2(\R^d)}=1$ (where we identify~${\mathfrak p}_\lambda$ with $\R^p$).

\begin{proposition}\label{prop:examples}
\begin{enumerate}

\item  Assume $\Phi_\eps=\Phi$ is independent of $\eps$ and 
$$\eps^2\lambda_\eps\Tend{\eps}{0}\lambda_0\not=0.$$ 
Then 
any semi-classical measure of $(v^\eps)_{\eps>0}$  is equivalent to the pair $(\Gamma,\gamma)$ with 
$$\Gamma=\mathcal F \Phi\otimes(\mathcal F \Phi)^*\;\;{\rm and}\;\;  d\gamma(x,\lambda)=\left(|a(x_{\mathfrak z})|^2 dx_{\frak z} \,\otimes\delta_{x_{\mathfrak v}=0}\right)\otimes \delta_{\lambda=\lambda_0}.$$

\item  
Assume  $\Phi_\eps=h_{\alpha_\eps}$ a Hermite function with 
$\alpha_\eps \sim \eps^{-\alpha}$ with $0<\alpha<{1\over d+1}$. Assume  
$$ \eps^2|\lambda_\eps|  (2|\alpha_\eps|+d)  \Tend{\eps}{0}  \mu^2 \in \R^*_+.$$
 Then  any  semi-classical measure of $(v^\eps)_{\eps>0}$ is equivalent to the pair $(\Gamma,\gamma)$ 
 supported
  in $G_{\mathfrak z} \times \{\lambda=0\}$, so we may choose $\Gamma =1$.
 Furthermore, the radialisation $\tilde \gamma (x,\omega)=\int_{SO(\mathfrak v)} \gamma (x,k\omega)dk$ of $\gamma$ satisfies 
  	$$\tilde \gamma (x,\omega)=\left(|a(x_{\mathfrak z})|^2 dx_{\frak z} \,\otimes\delta_{x_{\mathfrak v}=0}\right)\otimes  \left( \delta_{\omega=\mu\varsigma}{\bf 1}_{\varsigma\in {\bf S}^{2d-1}}d\varsigma\right)$$
where $d\varsigma$ denotes the probability measure on the sphere ${\bf S}^{2d-1}$. 


\end{enumerate}
\end{proposition}

 Recall that  $\sigma( x,(0,\omega)) $ is the value of the map $\sigma (x,\cdot)$ defined on $\widehat G$ above the points $\pi^{0,\omega}$ of $G$ and that this value is a complex number. 
 
 \medskip 
 
 The first case above corresponds to a concentration on some $\lambda_0\in{\mathfrak z}^*\setminus\{0\}$ and the second one, to some concentration on the set of dimension~$1$ representations $\pi^{0,\omega}$, $\omega\in{\bf S}^{2d}$. 
The analysis relies on Lemma~\ref{lem:dimfinie} and  illustrates the  phenomena of convergence of infinite dimensioned representations towards finite dimensioned one, as  illustrated by the Heisenberg fan of Figure~1.  
 
\medskip 

Pushing the argument further and using properties of the bounded spherical functions given in Section~2, one can  show that the radialisation of the family $v^\eps$ used in Part~2 above admits the semi-classical measures 
$$\gamma (x,\omega)=\left(|a(x_{\mathfrak z})|^2 dx_{\frak z} \,\otimes\delta_{x_{\mathfrak v}=0}\right)\otimes  \left( \delta_{\omega=\mu\varsigma}{\bf 1}_{\varsigma\in {\bf S}^{2d-1}}d\varsigma\right)$$
supported above $G_{\mathfrak z} \times \{\lambda=0\}$.

\medskip

One can also replace the function $a$ by a concentrating family 
$$a_\eps(x_{\mathfrak z})=\eps^{-p\beta}a(\delta_{\eps^{-\beta} }((x_{\mathfrak z}^0)^{-1}x_{\mathfrak z}))$$
with $a\in{\mathcal S}(\R^p)$ and $\beta\in(0,1/2)$, we shall then obtain a Dirac mass in $x_{\mathfrak z}^0$ instead of the absolutely continuous measure $|a(x_{\mathfrak z}|^2 dx_{\mathfrak z}$. Indeed, in the proof below, we just use 
\begin{equation}\label{prop:aeps}
\forall Z\in{\mathfrak z},\;\;\forall \eps>0, \;\; \eps\|Za_\eps\|_{L^2(G_{\mathfrak z})}\Tend{\eps}{0} 0,
\end{equation}
which is guaranteed by the condition $\beta\in(0,1/2)$.
 It is likely that the proof of section 6.5 in~\cite{FF}, which is based on these convergences, adapt to the semi-classical setting for proving that any element of ${\mathcal M}_1^+(G\times \widehat G)$ is a semi-classical measure of some bounded family of $L^2(G)$.

\medskip 

Before going to the proof, let us first  detail basic facts about the family $v^\eps$:

\medskip

 \noindent $\bullet$ {\it The norm of $v^\eps$}. The sequence $v^\eps$ is a bounded family of $L^2(G)$ since we have
\begin{eqnarray*}
\| v^\eps\|^2_{L^2(G)}  &= &|\lambda_\eps|^{d}\int_G |a(x_{\mathfrak z})|^2 | \left(\pi^{\lambda_\eps}_{ x_{\mathfrak v}} \Phi_\eps, \Phi_\eps\right)| ^2 dx\\
& = &\| a\|^2_{L^2(\R^p)} |\lambda_\eps|^{d} \int_{G_{\mathfrak v}}  \left| \left(\pi^ {\hat \lambda_\eps}_{\sqrt{|\lambda_\eps|} x_{\mathfrak v}} \Phi_\eps, \Phi_\eps\right)\right| ^2 dx_{\mathfrak v}\\
& = & \| a \|^2_{L^2(\R^p)} \int_{G_{\mathfrak v}}  \left| \left(\pi^{ \hat \lambda_\eps}_{ x_{\mathfrak v}} \Phi_\eps, \Phi_\eps\right)\right| ^2 dx_{\mathfrak v}
\end{eqnarray*}
where $\hat\lambda_\eps={\lambda_\eps\over|\lambda_\eps|}$ is of modulus $1$. We then observe 
$$\displaylines{
 \int_{G_{\mathfrak v}} \left| \left(\pi^{ \hat \lambda_\eps}_{ x_{\mathfrak v}} \Phi_\eps, \Phi_\eps\right)\right| ^2 dx_{\mathfrak v}= \int_{p,q,\xi\in\R^d} {\rm e}^{i q\cdot (\xi-\xi')} \Phi_\eps(\xi+p)\overline{\Phi_\eps(\xi)} \Phi_\eps(\xi')\overline{\Phi_\eps(\xi'+p)}d\xi\,dp\,dqd\xi'
 = c_1\| \Phi_\eps\|^4_{L^2(\R^d)}\cr}$$
 where $c_1>0$ is a universal constant. We thus obtain
 $$\| v^\eps\|_{L^2(G)} =c_1\|\Phi_\eps\|^2_{L^2(\R^d)} \| a\|_{L^2(G_{\mathfrak z})}.$$
 
 \medskip
 
 \noindent $\bullet$ {\it The $\eps$-oscillation of $v^\eps$}. We have 
 $$\displaylines{
 \forall V\in{\mathfrak v},\;\; \| Xv^\eps\|_{L^2(G)}= O(\sqrt{|\lambda_\eps|}),\cr
 \forall Z\in{\mathfrak z},\;\;\|Zv^\eps\|_{L^2(G)}= O\left(|\lambda^\eps|+\| Za\|_{L^2(\R^p)}\right).\cr
 }$$
Therefore, the family $(v^\eps)$ is $\eps$-oscillating by Sobolev criteria of Proposition~\ref{prop:sobcri}.
Example (1) is strictly oscillating while example~(2) is not. 

\medskip

 \noindent $\bullet$ {\it The Fourier transform of $v^\eps$}.
$${\mathcal F}(v^\eps)(\lambda)= |\lambda_\eps|^{d/2}  B_\eps(\lambda)\widehat a\left(\lambda-\lambda_\eps\right)\;\;{\rm 
with}\;\; 
B_\eps(\lambda)=\int_{G_{\mathfrak v}} (\pi^{\lambda} _{x_{\mathfrak v} } )^* \left(\pi^{ \lambda_\eps} _{x_{\mathfrak v} }\Phi_\eps,\Phi_\eps\right) dx_{\mathfrak v} .$$
One can observe that 
$$|\lambda_\eps|^d B_\eps(\lambda_\eps) = \Phi_\eps\otimes\Phi_\eps^*.$$

\medskip

We describe the semi-classical measures of $(v^\eps)_{\eps>0}$ in the next subsection. 

\subsection{Proof of Proposition~\ref{prop:examples} }

The
beginning of the proofs for Parts (1) and (2) of  Proposition~\ref{prop:examples} is the same: it consists of the following lemma.
After its proof, we will  analyse each case separately.

 \begin{lemma}
 \label{lem_prop:examples} 
 Let $\sigma = \mathcal F(\kappa_x)\in \mathcal A_0$.
Let $\chi\in{\mathcal C}^\infty_0(G) $ be identically equal to $1$ close to $0$. 
Under the hypotheses of Section \ref{subsec_osc}, 
 we have
 $$
\left| \left({\rm Op}_\eps(\sigma)v^\eps,v^\eps\right) -
 \int _{G_{\mathfrak z}}
| a(x_{\mathfrak z})|^2
\left(  \sigma_\eps( x_{\mathfrak z},\eps^{2}\lambda_\eps) \Phi_\eps,\Phi_\eps\right) dx_{\mathfrak z}\right|
\leq C_{1,a,\eps} |\lambda_\eps|^{-1/2} \  + \ C_{2,a} \eps,
$$
where 
$C_{2,a}$ and $C_{1,a,\eps}$ are constant of the form
$$
C_{2,a} \leq C_2 \sup_{G_{\mathfrak z}} |a| 
\sup_{G_{\mathfrak z}} |\nabla a|
$$
with $C_2>0$ independent of $\eps$ and $a$, and
$$
C_{1,a,\eps} \leq C_{1,N} \|a\|_{L^2}^2
\max_{|\beta_1|+\ldots+|\beta_8|\leq N }
\|\xi^{\beta_1} \partial_\xi^{\beta_2} 
\Phi_{\eps}\|_{L^2}
\|\xi^{\beta_3} \partial_\xi^{\beta_4} 
\Phi_{\eps}\|_{L^2}
\|\xi^{\beta_5} \partial_\xi^{\beta_6} 
\Phi_{\eps}\|_{L^2}
\|\xi^{\beta_7} \partial_\xi^{\beta_8} 
\Phi_{\eps}\|_{L^2},
$$
with $C_{1,N}$ independent of $\eps$ and $a$, for any integer  $N\geq 2d+1$. 	
 \end{lemma}

\begin{proof}[Proof of Lemma \ref{lem_prop:examples}]
For every $\eps>0$, we set $\sigma_\eps(x,\lambda) = \mathcal F ( \kappa_x \chi(\delta_{\eps} \cdot )) (\lambda)$.
By Propositions~\ref{prop:diag} and~\ref{prop_L2bdd}, 
$$
\left({\rm Op}_\eps(\sigma)v^\eps,v^\eps\right)
=
\left({\rm Op}_\eps(\sigma_{\eps})v^\eps,v^\eps\right)
+O(\eps^N),
$$
for any $N>0$.
We compute easily with the changes of variables $\eps^2\lambda \rightarrow \lambda ,\;\; \delta_{\eps^{-1}}  (xy^{-1})\rightarrow w$
\begin{align*}
\left({\rm Op}_\eps(\sigma_{\eps})v^\eps,v^\eps\right)
&=
c_0 \int _{G\times G \times \widehat G}{\rm Tr} \left(\pi^\lambda_{y^{-1}\cdot x} \sigma_{\eps}(x,\eps^2\lambda)\right) v^\eps(y)\overline{v^\eps(x)} |\lambda|^d dx\,dy \,d\lambda
\\
&=
c_0\int _{G\times G\times \widehat G}{\rm Tr} \left(\pi^\lambda_{w} \sigma_{\eps}(x,\lambda) \right)v^\eps((\delta_\eps  w)^{-1}x)\overline{v^\eps(x)} |\lambda|^d dx\,dw \,d\lambda
\\
&=\int _{G\times G}\chi(\delta_\eps w)\kappa_x(w)v^\eps((\delta_\eps  w)^{-1}x)\overline{v^\eps(x)}  dx\,dw,
\end{align*}
having used the Fourier inversion formula.
The Taylor formula \cite[(1.41)]{follandstein} yields
$$ a((\delta_\eps w)^{-1}x)_{\mathfrak z})=a(x_{\mathfrak z}) +A_{\eps} (x,w),
$$
with 
$$
|A_{\eps}(x,w)| \lesssim \sum_{j=1}^{2d} \eps |w_j| 
\sup_{\|z\|\lesssim \|\delta_\eps w\|} |V_j a(xz)|.
$$ 
here $\|\cdot\|$ denotes a pseudo-norm used, for instance the one used in the proof of Proposition~\ref{prop:diag}.
We can now write
$$
\left({\rm Op}_\eps(\sigma_{\eps})v^\eps,v^\eps\right)
= E_\eps + R_\eps,
$$
where the first term is 
$$
E_\eps:= \int _{G\times G}\chi(\delta_\eps w)\kappa_x(w)
|\lambda_\eps|^d | a(x_{\mathfrak z})|^2\overline{\left(\pi^{\lambda_\eps}_x \Phi_\eps,\Phi_\eps\right)}\left(\pi^{\lambda_\eps}_{(\delta_\eps w)^{-1}x} \Phi_\eps,\Phi_\eps\right) dx\,dw,
$$
and the remainder is
$$
R_\eps:=\int _{G\times G} \chi(\delta_\eps w)\kappa_x(w) \overline{a(x_{\mathfrak z})}A_\eps  (x,w)|\lambda_\eps|^d   \overline{\left(\pi^{\lambda_\eps}_x \Phi_\eps,\Phi_\eps\right)}\left(\pi^{\lambda_\eps}_{(\delta_\eps w)^{-1}x} \Phi_\eps,\Phi_\eps\right)
dxdw.
$$ 
Note that $x$ and $(\delta_\eps w)^{-1}x$ are compactly supported on the support of the integral defining $R_\eps$. Therefore, we can find two compactly supported smooth functions $\chi_1$ and $\chi_2$ such that 
\begin{align*}
R_\eps=
\int _{G\times G} \chi(\delta_\eps w)\kappa_x(w) \overline{a(x_{\mathfrak z})}A_\eps  (x,w) \overline{f^\eps_1(x)}f^\eps_2((\delta_\eps w)^{-1} x) dxdw, 
\end{align*}
where 
$$f^\eps_j(x):= |\lambda_\eps|^{d/2} \chi_j(x_{\mathfrak z}) \left(\pi^{\lambda_\eps}_x \Phi_\eps,\Phi_\eps\right),
\qquad j=1,2.
$$
We observed that the two families $x\mapsto f^\eps_j(x)$ are  bounded in $L^2(G)$ uniformly with respect to $\eps>0$.
We therefore estimate the remainder in the following way, using first the estimate for $A_\eps$:
$$
|R_\eps|
\lesssim \eps \sum_{j=1}^{2d} \sup_{G_{\mathfrak z}} |a| \sup_K |V_j a| 
\int_{G\times G} |w_j| \sup_x |\kappa_x |(w) 
|f_1^\eps (w) | \ |f_2^\eps ((\delta_\eps w)^{-1} x) | dx dw,
$$
where $K$ is a compact containing 
$\{x_1x_2 : x_1\in {\rm supp}\, \chi, x_2 \in x-{\rm supp}\, \sigma(x,\cdot)\}$;
note that 
$\sup_K |V_j a| \lesssim \sup_{G_{\mathfrak z}} |\nabla a|.$
Using the change of variable $w'=\delta_\eps w$ and the Cauchy-Schwartz inequality, 
this last integral is estimated by
\begin{align*}
&\int_{G\times G} |w_j| \sup_x |\kappa_x |(w) 
|f_1^\eps (x) | \ |f_2^\eps ((\delta_\eps w)^{-1} x) | dx dw
\\&\qquad=
\eps^{-Q}\int_{G\times G} \eps^{-1}|w'_j| \sup_x |\kappa_x |(\delta_{\eps^{-1}}w') \ 
|f_1^\eps (x) | \ |f_2^\eps ({w'}^{-1} x) | \ dx dw	
\\&\qquad\leq \||f_1^\eps \|_{L^2} 
\eps^{-Q} \|(\eps^{-1}|w'_j| \sup_x |\kappa_x |(\delta_{\eps^{-1}}\cdot )) * f_2^\eps \|_{L^2}
\\&\qquad\leq \||f_1^\eps \|_{L^2} 
\eps^{-Q} \|\eps^{-1}|w'_j| \sup_x |\kappa_x |(\delta_{\eps^{-1}}w' )\|_{L^1(dw')}  \|f_2^\eps \|_{L^2}
\\&\qquad\qquad = \||f_1^\eps \|_{L^2}  \|f_2^\eps \|_{L^2}
\||w_j| \sup_x |\kappa_x(w) |\|_{L^1(dw)}. 
\end{align*}
undoing the change of variable $w'=\delta_\eps w$  after having used Young's convolution inequality.
Hence we have obtained
$$
|R_\eps| \lesssim \eps \  
\sup_{G_{\mathfrak z}} |a| \
\sup_{G_{\mathfrak z}} |\nabla a| \
\|f_1^\eps \|_{L^2}  \|f_2^\eps \|_{L^2}
\max_{j=1,\ldots,2d}
\||w_j| \sup_x |\kappa_x (w) |\|_{L^1(dw)}.
$$

We now concentrate on the main term $E_\eps$.
Writing 
$$\pi^{\lambda_\eps}_{(\delta_\eps w)^{-1}x}= \pi^{\lambda_\eps}_{(\delta_\eps w)^{-1}}\pi^{\lambda_\eps}_{x}= \pi^{\eps^{2}\lambda_\eps}_{w^{-1}}\pi^{\lambda_\eps}_{x}
\quad\mbox{and}\quad
\pi^{\lambda_\eps}_x= {\rm e}^{-i\lambda_\eps\cdot x_{\mathfrak z}} \pi_{\lambda_\eps} x_{\mathfrak v}
$$
the integration in $w$ gives 
\begin{align*}
E_\eps
& =  \int _{G}
|\lambda_\eps|^d | a(x_{\mathfrak z})|^2\overline{\left(\pi^{\lambda_\eps}_x \Phi_\eps,\Phi_\eps\right)}
\left(  \sigma_\eps(x,\eps^{2}\lambda_\eps) \pi^{\lambda_\eps}_{x} \Phi_\eps,\Phi_\eps\right) dx
\\
&=
\int _{G}
|\lambda_\eps|^d | a_\eps(x_{\mathfrak z})|^2\overline{\left(\pi^{\lambda_\eps}_{x_{\mathfrak v}} \Phi_\eps,\Phi_\eps\right)}
\left(  \sigma_\eps(x,\eps^{2}\lambda_\eps) \pi^{\lambda_\eps}_{x_{\mathfrak v}} \Phi_\eps,\Phi_\eps\right) dx_{\mathfrak v} dx_{\mathfrak z},
\\
&=
\int _{G}
 | a_\eps(x_{\mathfrak z})|^2\overline{\left(\pi^{\hat \lambda_\eps}_{x_{\mathfrak v}} \Phi_\eps,\Phi_\eps\right)}
\left(  \sigma_\eps((|\lambda_\eps|^{-1/2} x_{\mathfrak v})x_{\mathfrak z},\eps^{2}\lambda_\eps) \pi^{\hat \lambda_\eps}_{x_{\mathfrak v}} \Phi_\eps,\Phi_\eps\right) dx_{\mathfrak v} dx_{\mathfrak z},
\end{align*}
having made the change of variable $x_{\mathfrak v}\mapsto |\lambda_\eps|^{-1/2}  x_{\mathfrak v}$.
We write
$$
E_\eps =F_\eps +S_\eps,
$$
where the first term is
$$
F_\eps := \int _{G}
 | a_\eps(x_{\mathfrak z})|^2\overline{\left(\pi^{\hat \lambda_\eps}_{x_{\mathfrak v}} \Phi_\eps,\Phi_\eps\right)}
\left(  \sigma_\eps( x_{\mathfrak z},\eps^{2}\lambda_\eps) \pi^{\hat \lambda_\eps}_{x_{\mathfrak v}} \Phi_\eps,\Phi_\eps\right) dx_{\mathfrak v} dx_{\mathfrak z}	,
$$
and the remainder is
$$
S_\eps :=
\int_0^1 \int_G | a(x_{\mathfrak z})|^2\overline{\left(\pi^{\hat\lambda_\eps}_{x_{\mathfrak v}} \Phi_\eps,\Phi_\eps\right)}
\frac{d}{ds}
\left(\sigma_\eps\left( (s |\lambda_\eps|^{-1/2} x_{\mathfrak v})\ x_{\mathfrak z}\right),\eps^{2}\lambda_\eps) \pi^{\hat\lambda_\eps}_{x_{\mathfrak v}} \Phi_\eps,\Phi_\eps\right) dx\, ds.
$$
For the first term, we use the facts that we have for any 
$\Phi_1,\widetilde \Phi_1,\Phi_2,\widetilde \Phi_2 
\in \mathcal S (\mathbb R^d)$
$$
\mbox{when} \ |\lambda| =1, \quad \int_{G_{\mathfrak v}}\overline{ \left(\pi^{\lambda}_{x_{\mathfrak v}} \Phi_1,\Phi_2\right)} \left(\pi^{\lambda}_{x_{\mathfrak v}} \widetilde \Phi_1,\widetilde \Phi_2\right) dx_{\mathfrak v}= (\Phi_1,\widetilde\Phi_1)
(\widetilde \Phi_2, \Phi_2),$$
and also that $\| \Phi_\eps\|_{L^2}=1$ to  obtain
$$
F_\eps = \int _{G_{\mathfrak z}}
| a(x_{\mathfrak z})|^2
\left(  \sigma_\eps( x_{\mathfrak z},\eps^{2}\lambda_\eps) \Phi_\eps,\Phi_\eps\right) dx_{\mathfrak z}.	
$$
For the remainder, we first write
\begin{align*}
S_\eps 
&=
 \int_G \ |\lambda_\eps|^{-1/2} \int_0^1 | a(x_{\mathfrak z})|^2\overline{\left(\pi^{\hat \lambda_\eps}_{x_{\mathfrak v}} \Phi_\eps,\Phi_\eps\right)}
\left(x_{\mathfrak v}\cdot \partial_{x_{\mathfrak v}}  \sigma_\eps( s(|\lambda_\eps|^{-1/2} x_{\mathfrak v})x_{\mathfrak z},\eps^{2}\lambda_\eps) \pi^{\hat \lambda_\eps}_{x_{\mathfrak v}} \Phi_\eps,\Phi_\eps\right)  ds \, dx
\\
&= \int_{|x_{\mathfrak v}|>1}  \ + \  \int_{|x_{\mathfrak v}|\leq 1}
:= S_{1,\eps} + S_{0,\eps}.  
\end{align*}
For $S_{0,\eps}$, we easily obtain
$$
|S_{0,\eps}|\leq C\,  |\lambda_\eps|^{-1/2} \| a\|_{L^2}  \| \Phi_{\eps}\|_{L^2}^4.
$$
For $S_{1,\eps}$ we need the following observation.
For each $\hat \lambda_\eps$, we set $x_{\mathfrak v} = {\rm Exp} [P+Q]$, $P=\sum_{1\leq j\leq d} p_j P_j$ and  $Q=\sum_{1\leq j\leq d} q_j Q_j$ and we observe for $1\leq j\leq d$ and for any $\Phi\in \mathcal S(\mathbb R^d)$
$$\displaylines{
q_j \left(\pi^{\hat \lambda_\eps}_{x_{\mathfrak v}} \Phi,\Phi\right)= +i \left(\pi^{\hat \lambda_\eps}_{x_{\mathfrak v}} \partial_{\xi_j}\Phi,\Phi\right) + i \left(\pi^{\hat \lambda_\eps}_{x_{\mathfrak v}} \Phi,\partial_{\xi_j}\Phi\right),\cr
p_j \left(\pi^{\hat \lambda_\eps}_{x_{\mathfrak v}} \Phi,\Phi\right)=  \left(\pi^{\hat \lambda_\eps}_{x_{\mathfrak v}} (\xi_j\Phi),\Phi\right) - \left(\pi^{\hat \lambda_\eps}_{x_{\mathfrak v}} \Phi,(\xi_j\Phi)\right).\cr}
$$
By using this observation, we deduce that  $S_{1,\eps}$ may be written as a linear combination of terms of the form 
$$\displaylines{\qquad S^N_{1,\eps} :=|\lambda_\eps|^{-1/2} \int_0^1 \int_G {\bf 1}_{|x_{\mathfrak v}|>1} |x_{\mathfrak v}|^{-N}| a(x_{\mathfrak z})|^2\hfill\cr\hfill \times\, \overline{\left(\pi^{\widehat{\lambda_\eps}}_{x_{\mathfrak v}} \Phi_{\eps,1},\Phi_{\eps,2}\right)}
\left(x_{\mathfrak v}\cdot \partial_{x_{\mathfrak v}}   \sigma_\eps( s(|\lambda_\eps|^{-1/2} x_{\mathfrak v})x_{\mathfrak z},\eps^{2}\lambda_\eps) \pi^{\widehat{\lambda_\eps}}_{x_{\mathfrak v}} \Phi_{\eps,3},\Phi_{\eps,4}\right) dx\, ds\qquad\cr}$$
for each $N\in \N$ large enough ($N\geq 2d+1$) so that the integral is absolutely convergent and where the functions $\Phi_{\eps,j}$ are linear combination of terms obtained by successive derivations of $\Phi_\eps$ or by multiplication by powers of the coordinates of $\xi$ (at most $N$-times in total). Therefore, these functions are in $L^2$ and  there exists a constant $C>0$ such that 
$$|S^N_{1,\eps}|\leq C\,  |\lambda_\eps|^{-1/2} \| a\|_{L^2}^2  \| \Phi_{\eps,1}\|_{L^2}\, \| \Phi_{\eps,2}\|_{L^2}\, \| \Phi_{\eps,3}\|_{L^2}\, \| \Phi_{\eps,4}\|_{L^2}.$$
Combining all the estimates shows Lemma \ref{lem_prop:examples}.	
\end{proof}

We can now particularise our study to the three cases of Proposition \ref{prop:examples}.

\medskip 

$\bullet $ \emph{Case (1):} 
Let us choose $\Phi_\eps=\Phi$ (with $\| \Phi\|^2=1$) independent of $\eps$ and $\eps^2 \lambda_\eps\Tend{\eps}{0}\lambda_0$.  
Lemma \ref{lem_prop:examples} yields:
\begin{align*}
	\left({\rm Op}_\eps(\sigma)v^\eps,v^\eps\right)
&=
\int _{G_{\mathfrak z}}
| a(x_{\mathfrak z})|^2
\left(  \sigma_\eps( x_{\mathfrak z},\eps^2 \lambda_\eps) \Phi,\Phi\right) dx_{\mathfrak z}	
	+ (\eps) \\
&=
\int _{G_{\mathfrak z}}
| a(x_{\mathfrak z})|^2
\left(  \sigma( x_{\mathfrak z},\lambda_0) \Phi,\Phi\right) dx_{\mathfrak z} +o(1),
\end{align*}
by Lebesgue dominated convergence.
Part (1) follows.

\medskip

$\bullet $ \emph{Case (2):}
 Let us choose now 
$\Phi_\eps=h_{\alpha_\eps}$ a Hermite function with $\alpha_\eps$  and  $\lambda_\eps$ as in Part (2).
Well known properties of the Hermite function yield for any $N_0\in \mathbb N$ and $\alpha\in \mathbb N^d$
$$
\max_{|\beta_1|+|\beta_2|\leq N_0 }\|\xi^{\beta_1} \partial_\xi^{\beta_2} h_\alpha\|_{L^2}
\lesssim_{N_0} |\alpha|^{N_0/2}.
$$
With the notation of Lemma \ref{lem_prop:examples}, 
we compute
$$
C_{1,\eps,N} \leq C'_{1,N} |\alpha_\eps|^{N/2}, 
\quad \mbox{so that} \ C_{1,\eps,N} |\lambda_\eps|^{-1/2} \lesssim |\alpha_\eps|^{(N+1)/2} \eps \longrightarrow_{\eps\to 0} 0,
$$
 for $N=2d+1$ and with the decay required in Part (2).
Therefore we have
$$
\left({\rm Op}_\eps(\sigma)v^\eps,v^\eps\right) = 
 \int _{G_{\mathfrak z}}
| a(x_{\mathfrak z})|^2
\left(  \sigma( x_{\mathfrak z},\eps^{2}\lambda_\eps) \Phi_\eps,\Phi_\eps\right) dx_{\mathfrak z} +o(1).$$
The first consequence of this relation and of $\eps^2\lambda_\eps\rightarrow 0$  is that any semi-classical measure $\Gamma d\gamma$ of the family $(v^\eps)$ is supported in $\{(0,\omega),\;\omega\in{\mathfrak v}^*\}$. Indeed, if $\sigma(x,\lambda)$ is compactly supported in ${\mathfrak z}^*\setminus\{0\}$, then the principal term of the right-hand side above is $0$ as $\eps$ is small enough. We deduce that any semi-classical measure of $(v^\eps)$ is equivalent to 
$$\Gamma(x,\lambda) =1\;\;{\rm and}\;\;\gamma(x,\lambda)={\bf 1}_{\lambda=0} \gamma(x,(0,\omega)).$$
Besides, 
by Lemma~\ref{lem:dimfinie} and Remark~\ref{rem:magenta}, we have 
$$\int_{SO({\mathfrak v})}\left(\sigma_\eps^{(k)}(x_{\mathfrak z},\eps^2\lambda_\eps)h_{\alpha_\eps},h_{\alpha_\eps}\right)dk  \Tend{\eps}{0}\int_{|\varsigma|=1} \sigma(x_{\mathfrak z},(0,\mu\varsigma))d\varsigma,$$
where the symbol $\sigma^{(k)} $ is associated with the convolution kernel 
$$\kappa_x^{(k)}(z)= \kappa_x({\rm Exp} (kz_{\mathfrak v}+z_{\mathfrak z})),\;\; z= {\rm Exp} (z_{\mathfrak v}+z_{\mathfrak z}),\;\; k\in SO({\mathfrak v}).$$
We deduce 
  by Lebesgue dominated convergence, 
$$\int_{SO({\mathfrak v})}\left({\rm Op}_\eps(\sigma^{(k)})v^\eps,v^\eps\right) dk \Tend{\eps}{0} 
\int _{G_{\mathfrak z}}\int_{|\varsigma|=1}
| a(x_{\mathfrak z})|^2
  \sigma( x_{\mathfrak z},(0,\mu\varsigma)) d\varsigma dx_{\mathfrak z}.$$
 In view of
  $$\sigma^{(k)} (x,(0,\omega))= \sigma(x,(0,k\omega)),\;\;\forall \omega\in{\mathfrak v},$$
  the latter relation implies that the measure $\gamma(x,(0,\omega)$ satisffies 
  $$\int_{SO({\mathfrak v})}
  \int_{G\times {\mathfrak v}^*} \sigma(x_{\mathfrak z},(0,k\omega) d\gamma(x,(0,\omega)) dk=  
  \int _{G_{\mathfrak z}}\int_{|\varsigma|=1}
| a(x_{\mathfrak z})|^2
  \sigma( x_{\mathfrak z},(0,\mu\varsigma)) d\varsigma dx_{\mathfrak z}.$$
  This implies Part (2).
%
%
%


\section{Appendix A - Symbolic calculus}
\label{app_symbcal}

We prove here Proposition~\ref{prop:symbcal}. 

\begin{proof}  Let us consider the composition of the operators of symbols $\sigma_1(x,\lambda)$ and $\sigma_2(x,\lambda)$ and denote by $\kappa^\eps_x(z)$ its convolution kernel. 
This function $\kappa^\eps_x(z)$ can be expressed in terms of the convolution kernels $\kappa_{1,x}(z)$ and $\kappa_{2,x}(z)$ associated with $\sigma_1(x,\lambda)$ and $\sigma_2(x,\lambda)$ respectively. Indeed, we have 
$$\eps^{-Q} \kappa^\eps_x(\delta_{\eps^{-1}}(y^{-1} x))= \eps^{-2Q} \int_G \kappa_{1,x} (\delta_{\eps^{-1}} (v^{-1} x))\kappa_{2,v}(\delta_{\eps^{-1}} (y^{-1} v))dv.$$
Therefore, performing the change of variable $z=\delta_{\eps^{-1}} (y^{-1} x)$ and $u= \delta_{\eps^{-1}} (v^{-1} x)$, we obtain
$$\kappa^\eps_x(z)  = \int_G \kappa_{1,x}(u)\kappa_{2,x\delta_{\eps}  u^{-1}}(zu^{-1})du.$$
This function $\kappa^\eps_x(z) $ is  smooth and compactly supported in $x$ and Schwartz class in $z$, besides there exists a constant $C>0$ such that 
$$\int_G \sup_{x\in G} |\kappa^\eps_x(z)| dz \leq C,$$
which implies that the family  
$$\sigma^\eps(x,\lambda)= {\mathcal F}(\kappa_x^\eps(\cdot))(\lambda)$$
is a family of symbols of ${\mathcal A}_0$ which generates a bounded family of operators on ${\mathcal L}(L^2(G))$. We are now going to prove that the symbol $\sigma^\eps(x,\lambda)$ has an asymptotic expansion. For this, we perform a Taylor expansion of $u\mapsto\kappa_{2,x(\delta_\eps u^{-1})}(z)$:
$$\kappa_{2,x(\delta_\eps u^{-1})}(z)=\kappa_{2,x}(z) -\eps\sum_{1\leq j\leq 2d} v_j(u) V_j\kappa_{2,x}(z) +\eps^2  R_\eps(x,z,u)$$
where, according to the Taylor expansion result of Theorem~3.1.51 in~\cite{FR},  there exists constants $\eta,C>0$ such that 
$$\left| R_\eps(x,v,u)\right| \leq C \sum_{|\alpha|\leq 2,\;
[\alpha]>1}
 |u|^{[\alpha]}\sup_{z\in G} 
 \left| X^\alpha \kappa_{2,z}(v)\right|,$$
 where $[\alpha]$ is the homogeneous length of $\alpha$: 
 $$[\alpha]= \sum_{j=1}^{2d} \alpha_j + 2\sum_{j=2d+1}^{ 2d+p} \alpha_j\;\;{\rm  if }\;\; \alpha=(\alpha_1,\cdots,\alpha_{2d},\alpha_{2d+1},\cdots, \alpha_{2d+p}).$$
  Therefore, 
 $$\kappa^\eps_x(z)=\int_G \kappa_{1,x}(u)\kappa_{2,x}(zu^{-1})du- \eps 
 \sum_{1\leq j\leq 2d} \int_G v_j(u) \kappa_{1,x}(u)V_j\kappa_{2,x}(zu^{-1})du +\eps^2 r^\eps_x(z)$$
 and there exists a constant $C>0$ such that for all $\eps>0$
 $$\int_G\sup_{x\in G} |r^\eps_x(v)| dv \leq  \sum_{|\alpha|\leq 2,\;[\alpha]>1}
 \sup_{x\in G} \left(\int_G |u|^{[\alpha|}|\kappa_{1,x}(u)|du\right) \left(\int_G \sup_{x\in G} |X^\alpha\kappa_{2,x} (v)| dv\right)<\infty.$$
As a consequence, the operator with convolution kernel $\eps^{-Q} r^\eps_x(\delta_{\eps^{-1}}z)$ is uniformly bounded in ${\mathcal L}(L^2(\R^d))$.
Besides, 
$$\int_G \kappa_{1,x}(u)\kappa_{2,x}(zu^{-1})du= ( \kappa_{2,x} * \kappa_{1,x})(z)$$
and 
$${\mathcal F} ( \kappa_{2,x} * \kappa_{1,x})(\lambda)= {\mathcal F}(  \kappa_{1,x})(\lambda)\circ {\mathcal F}(  \kappa_{2,x})(\lambda)= \sigma_1(x,\lambda)\circ \sigma_2(x,\lambda).$$
Similarly, observing that 
$$\int_G v_j(u) \kappa_{1,x}(u)V_j\kappa_{2,x}(zu^{-1})du
= V_j\kappa_{2,x}*(v_j \kappa_{1,x})$$
and 
$${\mathcal F}(v_j\kappa_{1,x})(\lambda)=\Delta_{v_j} {\mathcal F} (\kappa_{1,x})(\lambda)= \Delta_{v_j} \sigma_1(x,\lambda),$$
we obtain 
$${\mathcal F}( V_j\kappa_{2,x}*(v_j \kappa_{1,x})(\lambda)={\mathcal F}(v_j \kappa_{1,x})(\lambda) \circ{\mathcal F}(V_j\kappa_{2,x})(\lambda)=
\Delta_{v_j} \sigma_1(x,\lambda) \circ V_j \sigma_2(x,\lambda).$$
The relations
$$\begin{pmatrix}p\\q\end{pmatrix}=M^\lambda v,\;\; \begin{pmatrix}\Delta_p\\\Delta_q\end{pmatrix}=M^\lambda \Delta_v,\;\;(M^\lambda)^{-1} =\, ^tM^{\lambda},$$
stated in section~\ref{sec:difference} allows to conclude and obtain the relation~(\ref{eq:composition}). 

\medskip 

The proof concerning the adjoint of the operator of symbol $\sigma(x,\lambda)$ is similar since its convolution kernel  $\kappa^{\eps,*}_x(z)$ is given by
$$ \kappa^{\eps,*}_{x}(z)=\overline \kappa_{x(\delta_\eps z)^{-1}}(z).$$ One can then use Taylor expansion and identify the first terms of the expansion in the same manner than in the preceding proof.
\end{proof}


\section{Appendix B : the states of the $C^*$-algebra ${\mathcal A}$}\label{app_states}

We give here a proof of Proposition~\ref{prop:states}, which is adapted form the proofs of Propositions~{5.15} and~{5.17} in~\cite{FF}.

 Given 
$(\gamma,\Gamma)\in
{\mathcal M}_1^+(G\times \widehat G)$
satisfying  \eqref{eq_prop_trGg}, 
 one checks easily that the linear form $\ell$ defined via \eqref{eq_prop_elltrGg} is a state of  ${\mathcal A}$.
Therefore, the proof consists in proving that any state can be represented by a unique $\gamma d\Gamma$.
 Let $\ell$ be a state of the $C^*$-algebras  ${\mathcal A}$.
  The GNS construction \cite[Proposition 2.4.4]{Dixmier_C*} yields a representation $\rho$ of ${\mathcal A}$ 
   on the Hilbert space $\widetilde {\mathcal{H}}_\ell := {\mathcal A}/ \{\sigma : \ell(\sigma \sigma^*)=0\}$ 
such that
$$
 \ell(\sigma) =\lim_{n\to +\infty} (\rho(\sigma) \xi_n ,\xi_n)_{\widetilde {\mathcal{H}_\ell}} \, , 
\quad\sigma\in {\mathcal A},
$$
where the sequence $(\xi_n)_{n\in \mathbb N}$ is the image of any approximate identity of  ${\mathcal A}$ via the canonical projection  ${\mathcal A}\mapsto\widetilde{ \mathcal{H}}_\ell$.
We then decompose  \cite[Theorem 8.6.6]{Dixmier_C*} 
the representation $\rho$ 
(taking into account the possible multiplicities) as 
$$
(\rho,\widetilde{\mathcal {H}}_\ell) \sim (\rho_1,\widetilde{\mathcal{H}}_1) \oplus 2(\rho_2,\widetilde{\mathcal{H}}_2) \oplus \ldots \oplus \aleph_0 (\rho_\infty,\widetilde{\mathcal{H}}_\infty),
$$
and each $\rho_r$, $r\in \N\cup\{\infty\}$, may be disintegrated as
$$
\rho_r \sim  \int_{\widehat{{\mathcal A}}}\zeta d\gamma_r(\zeta);
$$
furthermore,  the positive measures $\gamma_1, \gamma_2,\ldots, \gamma_\infty$ are mutually singular in $\widehat{{\mathcal A}}$. 
Consequently we can write $\xi \in \widetilde{\mathcal{H}_\ell}$ as
$$
\xi \sim (\xi_1,\xi_2,\ldots, \xi_\infty), \quad
\mbox{with}\ \xi_r=(\xi_{r,s})_{1\leq s\leq r} \ \mbox{for each} \ r\in \N
\cup\{\infty\}, \ \mbox{and}\ \xi_{r,s}\in \widetilde{\mathcal{H}_r}.
$$
Note that
$$
1=|\xi|_{\mathcal{H}_\ell}^2 
= \sum_{r\in \N
\cup\{\infty\}}
\sum_{s=1}^r |\xi_{r,s}|_{\widetilde {\mathcal{H}}_r}^2
\qquad\mbox{with}\qquad
|\xi_{r,s}|_{\mathcal{H}_r}^2
= \int_{\widehat{\mathcal A}}|\xi_{r,s}(\zeta)|_{\widetilde{\mathcal{H}}_\zeta}^2 d\gamma_r(\zeta).
$$
Since 
we have identifed $\widehat{{\mathcal A}}$ with $G\times \widehat G$:
$$
\rho_r \sim \int_{G\times\widehat G }(x,\lambda) d\gamma_r(x,\lambda),\qquad
\mathcal{H}_r \sim\int_{G\times\widehat G} {\mathcal H}_\lambda   ,\qquad
\sum_{r=1}^\infty 
\sum_{s=1}^r 
\int_{G\times \widehat G}
|\xi_{r,s}(x,\lambda)|_{{\mathcal H}_\lambda}^2  d\gamma_r(x,\lambda)=1.
$$
Hence $\Gamma_r:=\sum_{s=1}^r \xi_{r,s} \otimes ({\xi_{r,s}})^*$
is a $\gamma_r$-measurable field on $G\times \widehat G$  of positive trace-class operators of rank~$r$. 
We have obtained:
\begin{align*}
\ell(\sigma) 
&= (\rho(\sigma) \xi ,\xi)
=\sum_{r\in \N\cup\{\infty\}} 
\sum_{s=1}^r
\int_{G\times \widehat G}
(\sigma(x,\lambda)\xi_{r,s}(x,\lambda),\xi_{r,s}(x,\lambda))_{L^2(\R^d)}
d\gamma_r(x,\lambda)\\
&=\sum_{r\in \N\cup\{\infty\}} 
\int_{G\times \widehat G}
{\rm Tr} \left( \sigma(x,\lambda) \Gamma_r (x,\lambda)\right)
d\gamma_r(x,\lambda).
\end{align*}
We now define 
the positive measure $\gamma:=\sum_r \gamma_r$.
As the measures $\gamma_r$ are mutually singular, 
the field
$\Gamma :=\sum_r \Gamma_r$ is measurable 
and satisfies
$$
\Gamma (x,\lambda)\geq 0, \quad{ \rm Tr} \left(\Gamma (x,\lambda)\right)<\infty,
\qquad
 \int_{G\times \widehat G} {\rm Tr}\,  \left(\Gamma (x,\lambda)\right) d\gamma(x,\lambda)=1 \, .
 $$
This shows the existence of the pair $(\Gamma, \gamma)$.

\medskip 

Let us now prove that this pair is unique up to the equivalence class and consider 
$(\gamma',\Gamma')
\in {\mathcal M}_1^+(G\times \widehat G)$
which also satisfies  \eqref{eq_prop_trGg} and \eqref{eq_prop_elltrGg}  for the same state $\ell$.
It suffices to consider the case of $\gamma$ and $\Gamma$ obtained as in the preceding argumentation and 
we may assume that $\gamma'$ and $\Gamma'$ have the same support in $G\times \widehat G$.
For each $r\in \N\cup\{\infty\}$,
let $B_r$ be the measurable subset of $G\times \widehat G$ where $\Gamma'(x,\lambda)$ is of rank $r$ a.e.
We may assume these subsets  disjoint.
We define the measure $\gamma'_r=1_{B_r}\gamma'$ 
and the field $\Gamma'_r:=1_{B_r}\Gamma'$
as the restrictions of $\gamma'$ 
and $\Gamma'$  to $B_r$.
As $\Gamma'_r$ is a measurable field of positive operators of rank $r$, 
there exists a measurable field of orthogonal vectors $(\xi_{r,s})_{s=1}^r$ such that 
$$\Gamma'_r=\sum_{s=1}^r\xi'_{r,s}\otimes( {\xi'_{r,s}})^*$$
and we have 
$${\rm Tr}\,  \Gamma'_r= \sum_{s=1}^r|\xi'_{r,s}|^2.$$
We define the representation $\rho'$ of ${\mathcal A}$
and the vector $\xi'$ of $\rho'$ via
$$
\rho':=\oplus_{r\in \N\cup\{\infty\}} r\int_{G\times \widehat G} (x,\lambda)\ d\gamma'_r(x,\lambda),
\qquad\mbox{and}\qquad
\xi':=\oplus_{r\in \N\cup\{\infty\}} \oplus_{s=1}^r \int_{G\times \widehat G}\xi'_{r,s}(x,\lambda) \ d\gamma'_r(x,\lambda).
$$
We observe that $\xi'$ is a unit vector:
$$
|\xi'|^2
=\sum_{r\in \N\cup\{\infty\}} \sum_{s=1}^r |\tilde \xi'_{r,s}|^2
=\sum_{r\in \N\cup\{\infty\}} \int_{G\times \widehat G}{\rm Tr}\, \Gamma'_r\ d\gamma'_r
= \int_{G\times \widehat G}{\rm Tr }\,\Gamma'\ d\gamma'=1.
$$
Moreover for any $\sigma\in {\mathcal A}$:
\begin{align*}
(\rho'(\sigma) \xi',\xi')
&=  \sum_{r\in \N\cup\{\infty\}} \sum_{s=1}^r
\int_{\Sigma_1} \left(\sigma\xi'_{r,s},\xi'_{r,s}\right) d\gamma'_r
=  \sum_{r\in \N\cup\{\infty\}} 
\int_{G\times \widehat G}{ \rm Tr }\, \left(\sigma \Gamma'_r\right) d\gamma'_r
\\&=
\int_{G\times \widehat G}{ \rm Tr } \,\left(\sigma \Gamma'\right) d\gamma' = \ell(\sigma).
\end{align*}
In other words, the state associated with $\rho'$ and $\xi'$ coincides with $\ell$. 
This implies  
 that $\rho'$ and $\rho$ are equivalent \cite[Proposition 2.4.1]{Dixmier_C*}, therefore the measures $\gamma'_r$ and $\gamma_r$ are equivalent for every $r\in \N\cup\{\infty\}$
 \cite[Theorem 8.6.6]{Dixmier_C*}.
 In other words, there exists a measurable positive function $f_r$ supported in $B_r$ such that 
 $$d\gamma'_r(x,\lambda)=f_r(x,\lambda) d\gamma_r(x,\lambda).$$
As  $\xi'$ corresponds to $\xi$ via the $(\rho',\rho)$-equivalence, 
we must have 
$$\Gamma_r(x,\lambda)=f_r(x,\lambda) \Gamma'_r(x,\lambda),$$
which concludes the proof.


\end{document}